\pdfoutput=1
\documentclass[english]{article}
\usepackage[T1]{fontenc}
\usepackage{geometry}
\usepackage{babel}
\usepackage{verbatim}
\usepackage{mathtools}
\usepackage{bm}
\usepackage{amsmath}
\usepackage{amssymb}
\usepackage[unicode=true,
 bookmarks=false,
 breaklinks=false,
 pdfborder={0 0 1},
 backref=section,
 colorlinks=false]{hyperref}

\makeatletter

\usepackage{babel}
\usepackage{bm}
\usepackage{enumitem}

\usepackage{babel}

\usepackage{algorithmic}

\usepackage{amsthm}
\usepackage{bbm}

\newcommand{\Xspace}{\ensuremath{\mathcal{X}}}

\newcommand{\Source}{\ensuremath{P}}
\newcommand{\Target}{\ensuremath{Q}}

\newcommand{\likeratio}{\ensuremath{\rho}}
\newcommand{\TruncLevel}{\ensuremath{{\tau_\numobs}}}
\newcommand{\TruncLike}[1]{\likeratio_\TruncLevel(#1)}

\newcommand{\TrueFun}{\ensuremath{f^\star}}
\newcommand{\FuncClass}{\ensuremath{\mathcal{F}}}

\newcommand{\noise}{\ensuremath{w}}

\newcommand{\numobs}{\ensuremath{n}}

\newcommand{\LRBound}{\ensuremath{B}}
\newcommand{\VarBound}{\ensuremath{V^2}}

\newcommand{\Hilbert}{\mathcal{H}}

\newcommand{\Event}{\mathcal{E}}

\newcommand{\Real}{\ensuremath{\mathbb{R}}}

\newcommand{\Exp}{\ensuremath{\mathbb{E}}}

\newcommand{\Prob}{\ensuremath{\mathbb{P}}}

\newcommand{\indicator}{\ensuremath{\mathbbm{1}}}

\newcommand{\PopCov}{\ensuremath{\bm{\Sigma}}}
\newcommand{\EmpCov}{\ensuremath{\widehat{\bm{\Sigma}}}}

\newcommand{\Term}{\ensuremath{T}}

\theoremstyle{plain}

\newtheorem{lemma}{\textbf{Lemma}}
\newtheorem{theorem}{\textbf{Theorem}}
\newtheorem{corollary}{\textbf{Corollary}}

\newtheorem{example}{\textbf{Example}}

\theoremstyle{definition}

\usepackage{mathrsfs}
\usepackage{overpic}

\usepackage[dvipsnames]{xcolor}
\definecolor{cm}{RGB}{0,0,200}
\definecolor{newpurple}{RGB}{200,0,200}
\newcommand{\cm}[1]{\textcolor{cm}{[CM: #1]}}
\newcommand{\rpcomment}[1]{\textcolor{magenta}{[RP: #1]}}

\newcommand{\LinearDim}{\ensuremath{d}}

\allowdisplaybreaks

\makeatother

\newcommand{\dualVar}{\ensuremath{\xi}}

\newcommand{\Quantity}{\ensuremath{\alpha}}

\newcommand{\corvec}{\ensuremath{v}}

\newcommand{\keyfunc}{\ensuremath{g}}

\newcommand{\coef}{\ensuremath{\beta}}

\newcommand{\uniVar}{\ensuremath{\omega}}

\newcommand{\thetavecRem}{\ensuremath{\thetavec_{\mathsf{R}}}}
\newcommand{\corvecRem}{\ensuremath{\corvec_{\mathsf{R}}}}
\newcommand{\MmatRem}{\ensuremath{\Mmat_{\mathsf{R}}}}

\newcommand{\Mfun}{\ensuremath{\mathcal{M}}}

\newcommand{\Bigsqrt}[1]{\ensuremath{ \Big \{ #1 \Big \}^{1/2}}}

\newcommand{\rlower}{\ensuremath{\delta_\numobs}}
\newcommand{\dlower}{\ensuremath{d_\numobs}}

\newcommand{\rcrit}{\ensuremath{\delta_\numobs}}

\newcommand{\Ccon}{\ensuremath{c_0}}

\newcommand{\thetavec}{\ensuremath{\theta}}
\newcommand{\thetatil}{\ensuremath{\tilde{\thetavec}}}
\let\thetatild\thetatil
\newcommand{\thetahat}{\ensuremath{\widehat{\theta}}}
\newcommand{\thetahaterm}{\thetahat_{\mathrm{\scriptsize{erm}}}}

\newcommand{\fhat}{\ensuremath{\widehat{f}}}
\newcommand{\fhatrw}{\ensuremath{\fhat^{\mathrm{\scriptsize{rw}}}}_{\lambda} }
\newcommand{\fhatkrr}{\ensuremath{\fhat_{\lambda}}}
\newcommand{\fhaterm}{\ensuremath{\fhat_{\mathrm{\scriptsize{erm}}}}}

\newcommand{\thetakrr}{\ensuremath{\thetahat}}
\newcommand{\firstcoord}{t}

\newcommand{\Termtwoa}{\ensuremath{\Term_{2a}}}
\newcommand{\Termtwob}{\ensuremath{\Term_{2b}}}

\newcommand{\ftil}{\ensuremath{\widetilde{f}}}

\newcommand{\Fclass}{\ensuremath{\mathcal{F}}}
\newcommand{\FclassStar}{\ensuremath{\Fclass^\star}}
\newcommand{\FuncClassStar}{\FclassStar}
\newcommand{\GoodFunc}{\mathcal{G}}

\newcommand{\KerFun}{\ensuremath{\mathscr{K}}}
\newcommand{\fstar}{\ensuremath{f^\star}}

\newcommand{\hilnorm}[1]{\ensuremath{\| #1 \|_{\Hilbert}}}

\newcommand{\htil}{\ensuremath{\widetilde{h}}}

\newcommand{\SUMJ}{\ensuremath{\sum_{j=1}^\infty}}

\newcommand{\unicon}{\ensuremath{c}}

\newcommand{\thetastar}{\ensuremath{\theta^\star}}
\newcommand{\hardfn}{\fstar_{\mathrm{hard}}}
\newcommand{\hardtheta}{{\thetastar_{\mathrm{hard}}}}

\newcommand{\Mmat}{\ensuremath{\bm{M}}}
\newcommand{\IdMat}{\ensuremath{\bm{I}}}
\newcommand{\LamMat}{\ensuremath{\bm{\Lambda}}}
\newcommand{\DelMat}{\ensuremath{\bm{\Delta}}}
\newcommand{\Vmat}{\ensuremath{\bm{V}}}

\newcommand{\Zmat}{\ensuremath{\bm{Z}}}

\newcommand{\diag}{\ensuremath{\mathsf{diag}}}
\newcommand{\trace}{\ensuremath{\mathsf{trace}}}
\newcommand{\KL}{\ensuremath{\mathsf{KL}}}
\newcommand{\1}{\mathbf{1}}
\newcommand{\matrixnorm}[2]{|\!|\!| #1 |\!|\!|_{{#2}}}
\newcommand{\opnorm}[1]{\ensuremath{\matrixnorm{#1}{2}}}

\newcommand{\Natural}{\ensuremath{\mathbb{N}}}

\newcommand{\Eset}{\ensuremath{\mathscr{E}}}

\newcommand{\ObsLike}{\ensuremath{\mathscr{L}}}
\newcommand{\Constraint}{\ensuremath{\mathcal{C}}}
\newcommand{\ConstraintPrime}{\ensuremath{\Constraint^\prime}}

\newcommand{\cstar}{\ensuremath{c^*}}

\newcommand{\real}{\ensuremath{\mathbb{R}}}

\newcommand{\mjwcomment}[1]{{\color{red}{MJW COMMENT: #1}}}

\newcommand{\rank}{\ensuremath{D}}

\newcommand{\radius}{\ensuremath{\|\fstar\|_{\Hilbert}}}

\newcommand{\upstairs}[1]{\textsuperscript{#1}}
\newcommand{\affilone}{\dag}
\newcommand{\affiltwo}{\ddag}

\newcommand{\affilchic}{$\diamond$}

\makeatletter
\long\def\@makecaption#1#2{
        \vskip 0.8ex
        \setbox\@tempboxa\hbox{\small {\bf #1:} #2}
        \parindent 1.5em  
        \dimen0=\hsize
        \advance\dimen0 by -3em
        \ifdim \wd\@tempboxa >\dimen0
                \hbox to \hsize{
                        \parindent 0em
                        \hfil 
                        \parbox{\dimen0}{\def\baselinestretch{0.96}\small
                                {\bf #1.} #2
                                } 
                        \hfil}
        \else \hbox to \hsize{\hfil \box\@tempboxa \hfil}
        \fi
        }
\makeatother

\newcommand{\widgraph}[2]{\includegraphics[keepaspectratio,width=#1]{#2}}

\newcommand{\Bias}{\ensuremath{\mathbf{b}}}
\newcommand{\Var}{\ensuremath{\mathbf{v}}}

\newcommand{\Exs}{\ensuremath{\mathbb{E}}}

\newcommand{\MyVarFun}{\ensuremath{\Var_\lambda(\LRBound)}}
\newcommand{\Ball}{\ensuremath{\mathcal{B}}}
\newenvironment{carlist}
 {\begin{list}{$\bullet$}
 {\setlength{\topsep}{0in} \setlength{\partopsep}{0in}
  \setlength{\parsep}{0in} \setlength{\itemsep}{\parskip}
  \setlength{\leftmargin}{0.15in} \setlength{\rightmargin}{0.08in}
  \setlength{\listparindent}{0in} \setlength{\labelwidth}{0.08in}
  \setlength{\labelsep}{0.1in} \setlength{\itemindent}{0in}}}
 {\end{list}}

\newcommand{\bcar}{\begin{carlist}}
\newcommand{\ecar}{\end{carlist}}

\begin{document}

\begin{center}
  {\bf{\LARGE{Optimally tackling covariate shift \\
   in RKHS-based nonparametric regression}}}

  \vspace*{.2in}

  \begin{tabular}{cc}
    Cong Ma\upstairs{\affilchic}, \quad Reese
    Pathak\upstairs{\affilone},\quad and \quad Martin
    J.\ Wainwright\upstairs{\affilone, \affiltwo} \\[1.5ex]
    \upstairs{\affilchic} Department of Statistics, University of
    Chicago \\ \upstairs{\affilone} Department of Electrical
    Engineering and Computer Sciences, UC Berkeley
    \\ \upstairs{\affiltwo} Department of Statistics, UC Berkeley \\
    \texttt{congm@uchicago.edu, \string{pathakr,wainwrig\string}@berkeley.edu}
  \end{tabular}
  \vspace*{.2in}

  \begin{abstract}
    We study the covariate shift problem in the context of
    nonparametric regression over a reproducing kernel Hilbert space
    (RKHS).  We focus on two natural families of covariate shift problems 
    defined using the likelihood ratios between the source and target distributions. When the 
    likelihood ratios are uniformly bounded, we prove that the 
    kernel ridge regression (KRR) estimator with a carefully chosen
    regularization parameter is minimax rate-optimal (up to a log factor)
   for a large family of RKHSs with regular kernel eigenvalues. Interestingly, 
   KRR does not require full
    knowledge of likelihood ratios apart from an upper bound on them.
    In striking contrast to
    the standard statistical setting without covariate shift, we also
    demonstrate that a na\"\i ve estimator, which minimizes the
    empirical risk over the function class, is strictly sub-optimal under 
    covariate shift as
    compared to KRR.  We then address the larger class of covariate
    shift problems where likelihood ratio is possibly unbounded yet has a
    finite second moment.  Here, we
    propose a reweighted KRR estimator that weights
    samples based on a careful truncation of the likelihood ratios.
    Again, we are able to show that this estimator is minimax rate-optimal,
    up to logarithmic factors. 

\end{abstract}

\end{center}

\section{Introduction}

A widely adopted assumption in supervised
learning~\cite{vapnik1999nature, gyorfi2002distribution} is that the
training and test data are sampled from the same distribution.  Such a
no-distribution-shift assumption, however, is frequently violated in
practice.  For instance, in medical image
analysis~\cite{guan2021domain, koh2021wilds}, distribution mismatch is
widespread across the hospitals due to inconsistency in medical
equipment, scanning protocols, subject populations, etc.  As another
example, in natural language
processing~\cite{jiang-zhai-2007-instance}, the training data are
often collected from domains with abundant labels (e.g., Wall Street
Journal), while the test data may well arise from a different domain
(e.g., arXiv which is mainly composed of scientific articles).

In this paper, we focus on a special and important case of
distribution mismatch, known as \emph{covariate shift}. In this
version, the marginal distributions over the input covariates may vary
from the training (or source) to test (or target) data\footnote{Hereafter,
we use source (resp.~target) and training (resp.~testing) interchangeably.}, while
the conditional distribution of the output label given the input
covariates is shared across training and testing. Motivating
applications include image, text, and speech classification in which
the input covariates determine the output
labels~\cite{sugiyama2012machine}.  Despite its importance in
practice, the covariate shift problem is underexplored in theory, when
compared to supervised learning without distribution mismatch---a
subject that has been well studied in the past decades~\cite{gyorfi2002distribution}.

This paper aims to bridge this gap by addressing several fundamental
theoretical questions regarding covariate shift. First, what is the
statistical limit of estimation in the presence of covariate shift?
And how does this limit depend on the ``amount'' of covariate shift
between the source and target distributions?  Second, does
nonparametric least-squares estimation---a dominant (and often
optimal) approach in the no-distribution-shift case---achieve the
optimal rate of estimation with covariate shift?  If not, what is the
optimal way of tackling covariate shift?


\subsection{Contributions and overview}

We address the aforementioned theoretical questions regarding
covariate shift in the context of nonparametric regression over
reproducing kernel Hilbert spaces
(RKHSs)~\cite{scholkopf2002learning}.  That is, we assume that under
both the source and target distributions, the regression function
(i.e., the conditional mean function of the output label given the
input covariates) belongs to an RKHS. In this paper, we focus on 
two broad families of source-target pairs 
depending on the configuration of the likelihood ratios between them.  

We first consider the uniformly $\LRBound$-bounded family in which the 
likelihood ratios are uniformly bounded by a quantity $\LRBound$. In this case, we 
present general performance upper bounds for the kernel ridge regression (KRR)
estimator in Theorem~\ref{ThmKRR}. Instantiations of this general bound 
on various RKHSs with regular eigenvalues are provided in Corollary~\ref{CorKRR}. 
It is also shown in Theorem~\ref{ThmLowerBound} that KRR---with an optimally 
chosen regularization parameter that
depends on the largest likelihood ratio~$\LRBound$---achieves the minimax lower
bound for covariate shift over this uniformly $\LRBound$-bounded family.
It is worth noting that the optimal regularization parameter shrinks as the likelihood
 ratio bound increases. 
 
 We further show---via a constructive argument---that the nonparametric
least-squares estimator, which minimizes the empirical risk on the
training data over the specified RKHS, falls short of achieving the
lower bound; see~Theorem~\ref{ThmConstrainedKRR}. This marks 
a departure from the
classical no-covariate-shift setting, where the constrained estimator
(i.e., the nonparametric least-squares estimator) and the regularized
estimator (i.e., the KRR estimator) can both
attain optimal rates of estimation~\cite{wainwright2019high}.  In essence, the failure arises
from the misalignment between the projections under the source and
target covariate distributions. Loosely speaking, nonparametric
least-squares estimation projects the data onto an RKHS according to
the geometry induced by the \emph{source} distribution. Under
covariate shift, the resulting projection can be extremely far away
from the projection under the \emph{target} covariate distributions.

In the second part of the paper, we turn to a more general setting, where the likelihood
ratios between the target and source distributions may not be bounded.
Instead, we only require the target and source covariate distributions
to have a likelihood ratio with bounded second moment. We propose a variant of KRR that weights 
samples based on a careful truncation of the
likelihood ratios.  We are able to show in
Theorem~\ref{ThmReweightedKRR} that this estimator is rate-optimal over
this larger class of covariate shift problems.

%
%

\subsection{Related work}

There is a large body of work on distribution mismatch and, in
particular, on covariate shift. Below we review the work that is
directly relevant to ours, and refer the interested reader to the
book~\cite{sugiyama2012machine} and the survey~\cite{pan2009survey}
for additional references.

Shimodaira~\cite{shimodaira2000improving} first studied the covariate
shift problem from a statistical perspective, and established the
asymptotic consistency of the importance-reweighted maximum likelihood
estimator (without truncation).  However, no finite-sample guarantees were 
provided therein. Similar to our work, Cortes and
coauthors~\cite{cortes2010learning} analyzed the importance-reweighted
estimator when the density ratio is either bounded or has a finite
second moment. However, their analysis applies to the function class
with finite pseudodimension (cf.~the
book~\cite{pollard2012convergence}), while the RKHS considered herein
does not necessarily obey this assumption. Moreover, even when the RKHS has a finite rank $\rank$, their result (e.g., Theorem 8) 
is sub-optimal---with a rate of $\sqrt{\VarBound \rank / n}$ compared to our optimal rate $\VarBound \rank / \numobs$. 
Here $\VarBound$ is the bound on the second moment of the likelihood ratios and $\numobs$ denotes the number of samples. 
 Recently, Kpotufe and
Martinet~\cite{kpotufe2021marginal} investigated covariate shift for
nonparametric classification. They proposed a novel notion called
transfer exponent to measure the amount of covariate shift between the
source and target distributions. An estimator based on $k$ nearest
neighbors was shown to be minimax optimal over the class of
covariate shift problems with bounded transfer exponent. Inspired by
the work of Kpotufe and Martinet, the current
authors~\cite{pathak2022new} proposed a more fine-grained similarity
measure for covariate shift and applied to nonparametric regression
over the class of H\"{o}lder continuous functions. It is worth
pointing out that both the transfer exponent and the new fine-grained
similarity measure are different and cannot directly be compared to the 
moment conditions we impose on the likelihood ratios in this work. In particular, 
there exist instances of covariate shift where the second moment of the likelihood 
ratios is bounded whereas the transfer exponent is infinite. One such case is 
when the source and target distributions are both Gaussian with the same mean 
but different scales. Another significant difference lies in the assumptions on the regression function. 
Kpotufe and
Martinet~\cite{kpotufe2021marginal} and Pathak et al.~\cite{pathak2022new} focused on 
the class of H\"{o}lder continuous functions, while we focus on RKHSs. This leads to drastically 
different optimal estimators.  
Schmidt-Hieber and Zamolodtchikov~\cite{schmidt2022local} recently established 
the local convergence of the nonparametric least-squares estimator for 
the specific class of 1-Lipschitz functions over the unit interval $[0,1]$ and applied 
it to the covariate shift setting. 
 
Apart from covariate shift, other forms of distribution mismatch have
been analyzed from a statistical perspective. Cai
et~al.~\cite{cai2021transfer} analyzed posterior shift and proposed an
optimal $k$-nearest-neighbor estimator.
Maity~et~al.~\cite{maity2020minimax} conducted the minimax analysis
for the label shift problem. Recently,
Reeve~et~al.~\cite{reeve2021adaptive} studied the general distribution
shift problem (also known as transfer learning) which allows both
covariate shift and posterior shift.

\paragraph{Notation.} Throughout the paper, we use $c, c', c_1,c_2$ to  
denote universal constants, which may vary from line to line.


\section{Background and problem formulation}
\label{sec:setup}

In this section, we formulate and provide background on the problem of
covariate shift in nonparametric regression.

\subsection{Nonparametric regression under covariate shift}

The goal of nonparametric regression is to predict a real-valued
response $Y$ based on a vector of covariates $X \in \Xspace$.  For
each fixed $x$, the optimal estimator in a mean-squared sense is given
by the regression function $\fstar(x) \coloneqq \Exs[Y \mid X = x]$.  In a
typical setting, we assume observations of $\numobs$ i.i.d.~pairs
$\{(x_i, y_i) \}_{i=1}^\numobs$, where each $x_i$ is drawn according
to some distribution $\Source$ over $\Xspace$, and then $y_i$ is drawn
according to the law $(Y \mid X = x_i)$.  We assume throughout that for each $i$, 
the residual $\noise_i \coloneqq y_i - \fstar(x_i)$ is a sub-Gaussian random variable 
with variance proxy $\sigma^2$. 

We refer to the distribution
$\Source$ over the covariate space as the \emph{source distribution}.
In the standard set-up, the performance of an estimator $\fhat$ is
measured according to its $L^2(\Source)$-error:
\begin{subequations}
\begin{align}
\|\fhat - \fstar\|^2_{\Source} & \coloneqq \Exs_{X \sim \Source} \big(
\fhat(X) - \fstar(X) \big)^2 \; = \; \int_{\Xspace} \big(\fhat(x) -
\fstar(x) \big)^2 p(x) dx,
\end{align}
where $p$ is the density of $\Source$.

In the covariate shift version of this problem, we have a different
goal---that is, we wish to construct an estimate $\fhat$ whose
$L^2(\Target)$-error is small.  Here the \emph{target distribution}
$\Target$ is different from the source distribution~$\Source$.  In
analytical terms, letting $q$ be the density of $\Target$, our goal is to
find estimators $\fhat$ such that
\begin{align}
  \|\fhat - \fstar\|^2_{\Target} & = \Exs_{X \sim \Target} \big(
  \fhat(X) - \fstar(X) \big)^2 \; = \; \int_{\Xspace} \big(\fhat(x) -
  \fstar(x) \big)^2 q(x) dx
\end{align}
\end{subequations}
is as small as possible.  Clearly, the difficulty of this problem
should depend on some notion of discrepancy between the source and
target distributions.
\subsection{Conditions on source-target likelihood ratios}

The discrepancy between the $L^2(\Source)$ and $L^2(\Target)$ norms is
controlled by the \emph{likelihood ratio}
\begin{align}
 \likeratio(x) \coloneqq \frac{q(x)}{p(x)},
\end{align}
which we assume exists for any $x \in \Xspace$.  By imposing different
conditions on the likelihood ratio, we can define different families
of source-target pairs $(\Source, \Target)$.  In this paper, we focus
on two broad families of such pairs.

\paragraph{Uniformly $\LRBound$-bounded families:}  For a
quantity $\LRBound \geq 1$, we say that the likelihood ratio is
\emph{$\LRBound$-bounded} if
\begin{align}
\label{EqnLRBound}  
\sup_{x \in \Xspace} \likeratio(x) \leq \LRBound.
\end{align}
It is worth noting that $\LRBound = 1$ recovers the case without
covariate shift, i.e., $\Source = \Target$.  Our analysis in
Section~\ref{SecLRBound} works under this condition.

\paragraph{$\chi^2$-bounded families:} A uniform bound on the likelihood
ratio is a stringent condition, so that it is natural to relax it.
One relaxation is to instead bound the second moment: in particular,
for a scalar $\VarBound \geq 1$, we say that the likelihood ratio is
\emph{$\VarBound$-moment bounded} if
\begin{align}
\label{EqnLRMoment}
\Exp_{X \sim \Source}[\likeratio^2(X)] & \leq \VarBound.
\end{align}
Note that when the uniform bound~\eqref{EqnLRBound}
holds, the moment bound~\eqref{EqnLRMoment} holds with $\VarBound
= \LRBound$. To see this, we can write $\Exp_{X \sim
  \Source}[\likeratio^2(X)] \; = \; \Exs_{X \sim \Target}
[\likeratio(X)] \; \leq \; \LRBound$.  However, the moment
bound~\eqref{EqnLRMoment} is much weaker in general.  It is also worth
noting that the $\chi^2$-divergence between $\Target$ and $\Source$
takes the form
\begin{align*}
\chi^2 (\Target || \Source)
= \Exp_{X \sim \Source}[\likeratio^2(X)] - 1.
\end{align*}
Therefore, one can understand the quantity $\VarBound - 1$ as an upper
bound on the $\chi^2$-divergence between $\Target$ and $\Source$.
Our analysis in Section~\ref{SecLRMoment} applies under this weaker
condition on the likelihood ratio.

\subsection{Unweighted versus likelihood-reweighted estimators}
\label{SecUnwRew}
In this paper, we focus on methods for nonparametric regression that
are based on optimizing over a Hilbert space $\Hilbert$ defined by a
reproducing kernel.  The Hilbert norm $\|f\|_\Hilbert$ is used as a
means of enforcing ``smoothness'' on the solution, either by adding a
penalty to the objective function or via an explicit constraint.

In the absence of any knowledge of the likelihood ratio, a na\"\i ve
approach is to simply compute the \emph{unweighted regularized
estimate}
\begin{align}
\label{EqnUnweightedKRR}
\fhatkrr & \coloneqq \arg \min_{f \in \Hilbert} \Big \{
\frac{1}{\numobs} \sum_{i=1}^\numobs (f(x_i) - y_i)^2 + \lambda
\hilnorm{f}^2 \Big \},
\end{align}
where $\lambda > 0$ is a user-defined regularization parameter.  When
$\Hilbert$ is a reproducing kernel Hilbert space (RKHS), then this
estimate is known as the \emph{kernel ridge regression} (KRR) estimate.
This is a form of empirical risk minimization, but in the presence of
covariate shift, the objective involves an empirical approximation to
$\Exs_\Source[(Y - f(X))^2]$, as opposed to $\Exs_\Target[(Y -
  f(X))^2]$.

If the likelihood ratio were known, then a natural proposal is to
instead compute the \emph{likelihood-reweighted regularized estimate}
\begin{align}
\label{EqnIntroReweightedKRR}  
\ftil^{\mbox{\scriptsize{rw}}}_\lambda & \coloneqq \arg \min_{f \in
  \Hilbert} \Big \{ \frac{1}{\numobs} \sum_{i=1}^\numobs
\likeratio(x_i) (f(x_i) - y_i)^2 + \lambda \hilnorm{f}^2 \Big \}.
\end{align}
The introduction of the likelihood ratio ensures that the objective
now provides an unbiased estimate of the expectation $\Exs_\Target[(Y
  - f(X))^2]$.  However, reweighting by the likelihood ratio also
increases variance, especially in the case of unbounded likelihood
ratios.  Accordingly, in Section~\ref{SecLRMoment}, we study a
suitably truncated form of the
estimator~\eqref{EqnIntroReweightedKRR}.

\subsection{Kernels and their eigenvalues}

Any reproducing kernel Hilbert space is associated with a positive
semidefinite kernel function $\KerFun: \Xspace \times \Xspace
\rightarrow \real$.  Under mild regularity conditions, Mercer's
theorem guarantees that this kernel has an eigen-expansion of the form
\begin{align}
\label{defn:kernel-function}
\KerFun(x, x') \coloneqq \SUMJ \mu_j \phi_j(x) \phi_j(x')
\end{align}
for a sequence of non-negative eigenvalues $\{\mu_j\}_{j \geq 1}$, and
eigenfunctions $\{\phi_j\}_{j \geq 1}$ taken to be orthonormal in
$L^2(\Target)$.  Given our goal of deriving bounds in the
$L^2(\Target)$-norm, it is appropriate to expand the kernel in
$L^2(\Target)$, as we have done here~\eqref{defn:kernel-function}, in
order to assess the richness of the function class.

Given the Mercer expansion, the squared norm in the reproducing kernel
Hilbert space takes the form
\begin{align*}
\|f\|_\Hilbert^2 & = \SUMJ \frac{\theta_j^2}{\mu_j}, \qquad \mbox{where
  $\theta_j \coloneqq \int_\Xspace f(x) \phi_j(x) q(x) dx$.}
\end{align*}
Consequently, the Hilbert space itself can be written as
\begin{align}
\Hilbert \coloneqq \bigg \{ f = \SUMJ \theta_{j} \phi_{j} \; \mid \;
\SUMJ \frac{ \theta_{j}^2 }{ \mu_{j} } < \infty \bigg \}.
\end{align}
Our goal is to understand the performance of nonparametric regression
under covariate shift when the regression function lies in $\Hilbert$.

Throughout this paper, we impose a standard boundedness condition on
the kernel function---namely, there exists some finite $\kappa > 0$ such
that 
\begin{align}
\label{EqnKerBound}  
\sup_{x \in \Xspace} \KerFun(x,x) & \leq \kappa^2.
\end{align}
Note that any continuous kernel over a compact domain satisfies this
condition.  Moreover, a variety of commonly used kernels, including
the Gaussian and Laplacian kernels, satisfy this condition over any
domain.


\section{Analysis for bounded likelihood ratios}
\label{SecLRBound}

We begin our analysis in the case of bounded likelihood ratios.  Our
first main result is to prove an upper bound on the performance of the
unweighted KRR estimate~\eqref{EqnUnweightedKRR}.  First, we prove a
family of upper bounds (Theorem~\ref{ThmKRR}) depending on the
regularization parameter $\lambda$.  By choosing $\lambda$ so as to
minimize this family of upper bounds, we obtain concrete results for
different classes of kernels (Corollary~\ref{CorKRR}).  We then turn
to the complementary question of lower bounds: in
Theorem~\ref{ThmLowerBound}, we prove a family of lower bounds that
establish that for covariate shift with $\LRBound$-bounded likelihood
ratios, the KRR estimator is minimax-optimal up to logarithmic factors
in the sample size.  This optimality guarantee is notable since it
applies to the unweighted estimator that does not involve full
knowledge of the likelihood ratio (apart from an upper bound).

In the absence of covariate shift, it is well-known that performing
empirical risk minimization with an explicit constraint on the
function also leads to minimax-optimal results.  Indeed, without
covariate shift, projecting an estimate onto a convex constraint set
containing the true function can never lead to a worse result.  In
Theorem~\ref{ThmConstrainedKRR}, we show that this natural property is
no longer true under covariate shift: performing empirical risk
minimization over the smallest Hilbert ball containing $\fstar$ can be
sub-optimal.  Optimal procedures---such as the regularized KRR
estimate---are actually operating over Hilbert balls with radius
substantially larger than the true norm $\|\fstar\|_\Hilbert$.

\subsection{Unweighted kernel ridge regression is near-optimal}
\label{SecUnweightedKRR}

We begin by deriving a family of upper bounds on the kernel ridge
regression estimator~\eqref{EqnUnweightedKRR} under covariate shift.
In conjunction with our later analysis, these bounds will establish
that the KRR estimate is minimax-optimal up to logarithmic factors for
covariate shift with bounded likelihood ratios.

\begin{theorem}
\label{ThmKRR}
Consider a covariate-shifted regression problem with likelihood ratio
that is $\LRBound$-bounded~\eqref{EqnLRBound} over a Hilbert space
with a $\kappa$-uniformly bounded kernel~\eqref{EqnKerBound}.  Then
for any $\lambda \geq 10 \kappa^2 / \numobs$, the KRR estimate
$\fhatkrr$ satisfies the bound
\begin{align}
\label{EqnKRRBound}
\| \fhatkrr - \TrueFun \|_{\Target}^2 \leq \underbrace{\vphantom{\sum_{j=1}^{\infty}} 4 \lambda
  \LRBound \radius^2}_{\Bias^2_\lambda(\LRBound)} + \underbrace{80 \sigma^2
  \LRBound \frac{\log \numobs}{\numobs} \sum_{j=1}^{\infty}
  \frac{\mu_j }{ \mu_j + \lambda \LRBound }}_{\MyVarFun}
\end{align}
with probability at least $1 - 28 \; \tfrac{\kappa^2}{ \lambda } e^{-\frac{
    \numobs \lambda }{16 \kappa^2}} - \frac{1}{\numobs^{10}}$. 
\end{theorem}

\noindent See Section~\ref{sec:proof-KRR} for the proof of this theorem.
In Appendix~\ref{sec:cor-KRR-bounded-upper}, we also present a corollary
which provides a corresponding expectation bound for the KRR estimator $\fhatkrr$
for such $\LRBound$-bounded covariate shifts. 

\medskip 
Note that the upper bound~\eqref{EqnKRRBound} involves two terms. The
first term $\Bias^2_\lambda(\LRBound)$ corresponds to the squared bias
of the KRR estimate, and it grows proportionally with the
regularization parameter $\lambda$ and the likelihood ratio bound
$\LRBound$. The second term $\MyVarFun$ represents the variance of the
KRR estimator, and it shrinks as $\lambda$ increases, so that
$\lambda$ controls the bias-variance trade-off.  This type of
trade-off is standard in nonparametric regression; what is novel
of interest to us here is how the shapes of these trade-off curves
change as a function of the likelihood ratio bound $\LRBound$.

Figure~\ref{FigTradeOff} plots the right-hand side of the upper
bound~\eqref{EqnKRRBound} as a function of $\lambda$ for several
different choices $\LRBound \in \{1, 5, 10, 15 \}$.  (In all cases, we
fixed a kernel with eigenvalues decaying as $\mu_j = j^{-2}$, sample
size $\numobs = 8000$, and noise variance $\sigma^2 = 1$.)  Of
interest to us is the choice $\lambda^*(\LRBound)$ that minimizes this
\begin{figure}[t]
  \begin{center}
    \widgraph{0.5\textwidth}{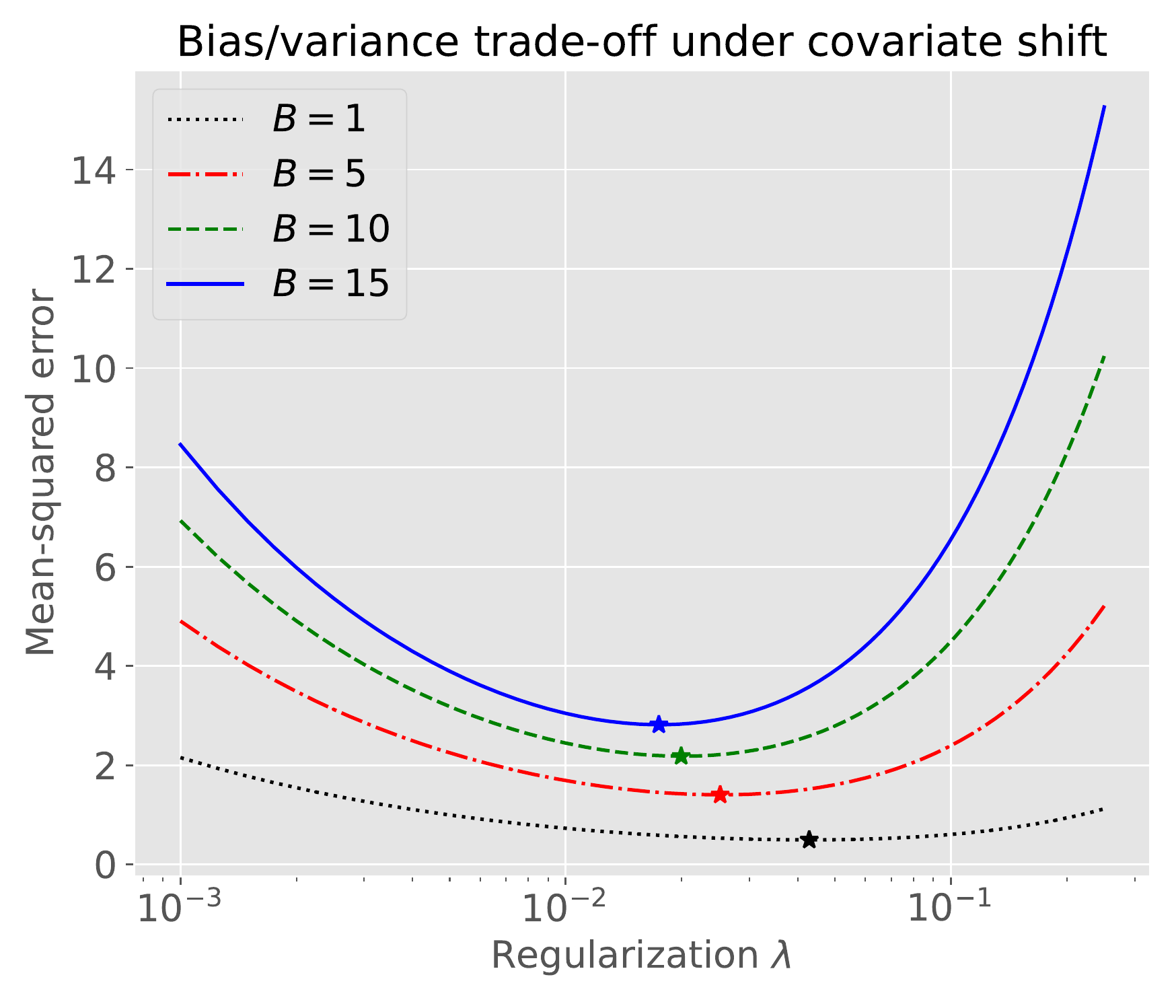}
    \caption{Plot of the upper bound~\eqref{EqnKRRBound} on the
      mean-squared error versus the log regularization parameter $\log
      \lambda$ for four different choices of the likelihood ratio
      bound $\LRBound$, in all cases with eigenvalues $\mu_j =
      (1/j)^2$, noise variance $\sigma^2 = 1$ and sample size $\numobs
      = 8000$.  The points marked with $\star$ on each curve
      corresponds to the choice of $\lambda^*(\LRBound)$ that
      minimizes the upper bound. Note how this minimizing value shifts
      to the left as $\LRBound$ increases above the standard problem
      without covariate shift ($\LRBound = 1$).\label{FigTradeOff}}
  \end{center}
  
\end{figure}
upper bound;  note how this optimizing $\lambda^*(\LRBound)$
shifts leftwards to smaller values as $\LRBound$ is increased.

We would like to understand the balancing procedure that leads to an
optimal $\lambda^*(\LRBound)$ in analytical terms.  This balancing
procedure is most easily understood for kernels with \emph{regular
eigenvalues}, a notion introduced in past
work~\cite{yang2017randomized} on kernel ridge regression.  For a
given targeted error level $\delta > 0$, it is natural to consider the
first index $d(\delta)$ for which the associated eigenvalue drops
below $\delta^2$---that is, $d(\delta) \coloneqq \min \{ j \geq 1 \mid
\mu_j \leq \delta^2 \}$.  The eigenvalue sequence is said to be
regular if\footnote{In fact, we can relax this to only require the minimizing 
$\delta$ in equation~\eqref{EqnKRRBoundRegular} to obey the tail bound.}
\begin{align}
  \label{EqnRegular}
  \sum_{j = d(\delta) +1}^\infty \mu_j & \leq c \, d(\delta) \delta^2
\end{align}
holds for some universal constant $c>0$. 
The class of kernels with regular eigenvalues includes kernels of finite-rank and those
with various forms of polynomial or exponential decay in their
eigenvalues; all are widely used in practice. 
For kernels with regular eigenvalues, the
bound~\eqref{EqnKRRBound} implies that there is a universal constant
$c'$ such that
\begin{align}
\label{EqnKRRBoundRegular}
\| \fhatkrr - \TrueFun \|_{\Target}^2 \leq c' \Big \{
\delta^2 \radius^2 +
\sigma^2 \LRBound \frac{d(\delta) \log \numobs}{\numobs} \Big \}
\qquad \mbox{where $\delta^2 = \lambda \LRBound$.}
\end{align}
We verify this claim as part of proving Corollary~\ref{CorKRR} below.

This bound~\eqref{EqnKRRBoundRegular} enables us to make (near)-optimal
choices of $\delta$---and hence $\lambda = \delta^2/\LRBound$.  Let us
summarize the outcome of this procedure for a few kernels of interest.
In particular, we say that a kernel has \emph{finite rank $\rank$} if
the eigenvalues $\mu_{j} = 0$ for all $j > \rank$.  The kernels that
underlie linear regression and polynomial regression more generally
are of this type.  A richer family of kernels has eigenvalues that 
exhibit \emph{$\alpha$-polynomial decay} $\mu_j \leq
c \, j^{-2 \alpha}$ for some $\alpha > 1/2$.  This kind of eigenvalue
decay is seen in various types of spline and Sobolev kernels, as well
as the Laplacian kernel.  It is easy to verify that both of these
families have regular eigenvalues. To simplify the presentation, 
we assume $\radius = 1$. 

\begin{corollary}[Bounds for specific kernels]
\label{CorKRR}
\begin{enumerate}
\item[(a)] For a kernel with rank $\rank$, as long as $\sigma^2 \rank \log \numobs 
\geq 10 \kappa^2$, the choice $\lambda =
  \frac{\sigma^2 \rank \log \numobs}{\numobs}$ yields an estimate
  $\fhatkrr$ such that
\begin{subequations}  
  \begin{align}
    \label{EqnFiniteUpper}
    \| \fhatkrr - \TrueFun \|_{\Target}^2 \leq c \sigma^2 \LRBound
    \frac{\rank \log \numobs}{\numobs}
\end{align}
with high probability.
\item[(b)] For a kernel with $\alpha$-decaying eigenvalues, suppose 
that $\sigma^2$ is sufficiently large so that  
  $\lambda = \linebreak \LRBound^{- \frac{1}{2 \alpha+1}} \; \big(\frac{\sigma^2
  \log \numobs}{\numobs} \big)^{\frac{2 \alpha}{2 \alpha + 1}} \geq 
  10 \kappa^2 / \numobs$. 
  Then the estimate $\fhatkrr$ obeys
  \begin{align}
\label{EqnAlphaUpper}    
\| \fhatkrr - \TrueFun \|_{\Target}^2 \leq c \Big( \frac{\sigma^2
  \LRBound \log \numobs}{\numobs} \Big)^{\frac{2 \alpha}{2 \alpha +
    1}}
  \end{align}
\end{subequations}
  with high probability.
\end{enumerate}
\end{corollary}
\begin{proof}
We begin by proving the upper bound~\eqref{EqnKRRBoundRegular}.  With
the shorthand $\delta^2 = \lambda \LRBound$, the variance term in our
bound~\eqref{EqnKRRBound} can be bounded as
\begin{align*}
\frac{1}{80} \MyVarFun = \sigma^2 \LRBound \frac{\log
  \numobs}{\numobs} \sum_{j=1}^{\infty} \frac{\mu_j }{\mu_j +
  \delta^2} & \leq \sigma^2 \LRBound \frac{\log \numobs}{\numobs} \Big
\{ \sum_{j=1}^{d(\delta)} 1 + \sum_{j > d(\delta)}^\infty
\frac{\mu_j}{\mu_j + \delta^2} \Big \},
\end{align*}
where, by the definition of $d(\delta)$, we have split the eigenvalues
into those that are larger than $\delta^2$, and those that are smaller
than $\delta^2$.  By the definition of a regular
kernel, the second term can be upper bounded
\begin{align*}
\sum_{j > d(\delta)}^\infty \frac{\mu_j}{\mu_j + \delta^2} & \leq
\frac{1}{\delta^2} c' \, d(\delta) \delta^2 \; = \; c' \, d(\delta).
\end{align*}
Putting together the pieces yields $\frac{1}{80} \MyVarFun \leq c_2
\sigma^2 \LRBound \frac{\log \numobs}{\numobs} d(\delta)$, for some
universal constant $c_2$.  Combining with the bias term yields the
claim~\eqref{EqnKRRBoundRegular}.

We now prove claims~\eqref{EqnFiniteUpper} and~\eqref{EqnAlphaUpper}.
For a finite-rank kernel, using the fact that $d(\delta) \leq \rank$ for any 
$\delta > 0$, we can set $\lambda = \frac{\sigma^2 \rank \log
  \numobs}{\numobs}$ to obtain the claimed bound~\eqref{EqnFiniteUpper}.
Now suppose that the kernel has $\alpha$-polynomial decay---that is,
$\mu_j \leq c j^{-2 \alpha}$ for some $c > 0$.  For any $\delta > 0$,
we then have $d(\delta) \leq c' \, (1/\delta)^{1/\alpha}$, and hence
\begin{align*}
\delta^2 + \sigma^2 \LRBound \frac{d(\delta) \log \numobs}{\numobs} &
\leq \delta^2 + c ' \sigma^2 \LRBound \frac{\log \numobs}{\numobs}
\big(\frac{1}{\delta}\big)^{1/\alpha}.
\end{align*}
By equating the two terms, we can solve for near-optimal $\delta$: in
particular, we set $\delta^2 = \Big(\frac{\sigma^2 \LRBound \log
  \numobs}{\numobs} \Big)^{\frac{2 \alpha}{2 \alpha + 1}}$ to obtain
the claimed result.  Notice that this choice of $\delta^2$ corresponds
to
\begin{align*}
  \lambda = \delta^2/\LRBound \; = \; \LRBound^{- \frac{1}{2
      \alpha+1}} \; \big(\frac{\sigma^2 \log \numobs}{\numobs}
  \big)^{\frac{2 \alpha}{2 \alpha + 1}},
\end{align*}
as claimed in the corollary.
\end{proof}


\subsection{Lower bounds with covariate shift for regular kernels}

Thus far, we have established a family of upper bounds on the
unweighted KRR estimate, and derived concrete results for various
classes of regular kernels.  We now establish that, for the class of
regular eigenvalues, the bounds achieved by the unweighted KRR
estimator are minimax-optimal.  Recall the definition $d(\delta)
= \min \{ j \geq 1 \mid \mu_j \leq \delta^2 \}$, and the
notion of regular eigenvalues~\eqref{EqnRegular}.  For a Hilbert space
$\Hilbert$, we let $\Ball_\Hilbert(1)$ denote the Hilbert norm ball of
radius one.
\begin{theorem}
\label{ThmLowerBound}
For any $\LRBound \geq 1$, there exists a pair $(\Source, \Target)$
with $\LRBound$-bounded likelihood ratio~\eqref{EqnLRBound} and an
orthonormal basis $\{ \phi_j \}_{j \geq 1}$ of $L^2(\Target)$ such
that for any regular sequence of kernel eigenvalues $\{\mu_j\}_{j \geq
  1}$, we have
\begin{align}
\label{EqnLowerBound}  
\inf_{\fhat} \sup_{ \TrueFun \in \Ball_\Hilbert(1) } \Exp [ \| \fhat -
  \TrueFun \|_{\Target}^{2} ] & \geq c \; \inf_{\delta > 0} \Big \{
\delta^2 + \sigma^2 \LRBound \frac{d(\delta)}{\numobs} \Big \}.
\end{align}
\end{theorem}
\noindent See Appendix~\ref{sec:proof-lower-bound} for the proof of
this claim. \\


Comparing the lower bound~\eqref{EqnLowerBound} to our achievable
result~\eqref{EqnKRRBoundRegular} for the unweighted KRR estimate, we
see that---with an appropriate choice of the regularization parameter
$\lambda$---the KRR estimator is minimax optimal up to a $\log
\numobs$ term.  In particular, it is straightforward to derive the
following consequences of Theorem~\ref{ThmLowerBound}, which parallel
the guarantees in Corollary~\ref{CorKRR}:
\bcar
\item For a finite-rank kernel, the minimax risk for
  $\LRBound$-bounded covariate shift satisfies the lower bound
  \begin{align*}
\inf_{\fhat} \sup_{ \TrueFun \in \Ball_\Hilbert(1) } \Exp [ \| \fhat -
  \TrueFun \|_{\Target}^{2} ] & \geq c \; \sigma^2 \LRBound
\frac{\rank}{\numobs}.
  \end{align*}
\item For a kernel with $\alpha$-polynomial eigenvalues, the
  minimax risk for $\LRBound$-bounded covariate shift satisfies
  the lower bound
  \begin{align*}
\inf_{\fhat} \sup_{ \TrueFun \in \Ball_\Hilbert(1) } \Exp [ \| \fhat -
  \TrueFun \|_{\Target}^{2} ] & \geq c \; \Big( \frac{\sigma^2
  \LRBound}{\numobs} \Big)^{\frac{2 \alpha}{2 \alpha + 1}}.
  \end{align*}
  \ecar
Note that both of these minimax lower bounds reduce to the known lower
bounds~\cite{yang2017randomized} in the case of no covariate shift
(i.e., $\LRBound = 1$).



\subsection{Constrained kernel regression is sub-optimal}

In the absence of covariate shift, procedures based on empirical risk
minimization with explicit constraints are also known to be
minimax-optimal.  In the current setting, one such estimator is the
constrained kernel regression estimate 
\begin{align}
\label{EqnConstrainedKRRForm}
\fhaterm \coloneqq \arg \min_{ f \in \Ball_\Hilbert(1) } \Big \{ \frac{1}{
  \numobs } \sum_{i=1}^\numobs (f(x_i) - y_i)^2 \Big \}.
\end{align}
Without covariate shift and for any regular kernel, this constrained
empirical risk minimization procedure is minimax-optimal over all
functions $\fstar$ with $\|\fstar\|_\Hilbert \leq 1$.

In the presence of covariate shift, this minimax-optimality turns out
to be false.  In particular, suppose that the eigenvalues decay as
$\mu_j = (1/j)^2$, so that our previous results show that the minimax
risk for $\LRBound$-bounded likelihood ratios scales as
$\big(\frac{\LRBound \sigma^2}{\numobs} \big)^{2/3}$.  It turns out
that there exists $\LRBound$-bounded pair $(\Source, \Target)$ and an
associated kernel class with the prescribed eigendecay for which the
constrained estimator~\eqref{EqnConstrainedKRRForm} is sub-optimal for a
broad range of $(\LRBound, \numobs)$ pairs.  In the following
statement, we use $c_1, c_2$ to denote universal
constants.
\begin{theorem}
  \label{ThmConstrainedKRR}
Assume that $\radius = 1$, and that $\sigma^2 =1$. 
For any $\LRBound \in [c_1 (\log \numobs)^2, c_2 \numobs^{2/3}]$,
there exists a $\LRBound$-bounded pair $(\Source, \Target)$ and RKHS
with eigenvalues $\mu_j \leq (1/j^2)$ such that
\begin{align}
\label{EqnConstrainedKRR}  
\sup_{\TrueFun \in \Ball_\Hilbert(1)} \Exp \Big[\| \fhaterm -
  \TrueFun \|_{\Target}^{2}\Big] \geq c_3 \,
\frac{\LRBound^{3}}{\numobs^{2}}.
 \end{align}
\end{theorem}
\noindent See Appendix~\ref{sec:proof-failure} for the proof of this
negative result. \\

In order to appreciate some implications of this theorem, suppose that
we use it to construct ensembles with $\LRBound_\numobs \asymp
\numobs^{2/3}$.  The lower bound~\eqref{EqnConstrainedKRR} then
implies that over this sequence of problems,
the maximal risk of $\fhaterm$ is bounded below by a universal constant.  On the
other hand, if we apply the unweighted KRR procedure, then we obtain
consistent estimates, in particular with $L^2(\Target)$-error that
decays as
\begin{align*}
\Big( \frac{\LRBound_\numobs}{\numobs} \Big)^{2/3} & =
\Big(\frac{\numobs^{2/3}}{\numobs} \Big)^{2/3} \; = \; \numobs^{-2/9}.
\end{align*}

\medskip

It is worth understanding why the constrained form of KRR is
sub-optimal, while the regularized form is minimax-optimal.  Recall
from Corollary~\ref{CorKRR} that achieving minimax-optimal rates with
 KRR requires particular choices of the regularization
parameter $\lambda^*(\LRBound)$, ones that decrease as $\LRBound$
increases (see Figure~\ref{FigTradeOff}).  This behavior suggests that
the Hilbert norm $\|\fhatkrr\|_\Hilbert$ of the regularized KRR
estimate with optimal choice of $\lambda$ should grow significantly
above $\|\fstar\|_\Hilbert = 1$ when we apply this method.

In order to confirm this intuition, we performed some illustrative
simulations over the ensembles, indexed by the pair $(\LRBound,
\numobs)$, that underlie the proof of Theorem~\ref{ThmConstrainedKRR};
see Appendix~\ref{sec:proof-failure} for the details.  With $\sigma^2
= 1$ remaining fixed, for each given pair $(\LRBound, \numobs)$, we
simulated regularized kernel ridge regression with the choice $\lambda
= 4^{2/3} \numobs^{-2/3} \LRBound^{-1/3}$, as suggested by
Corollary~\ref{CorKRR}.
\begin{figure}[t]
\begin{center}
\widgraph{0.5\textwidth}{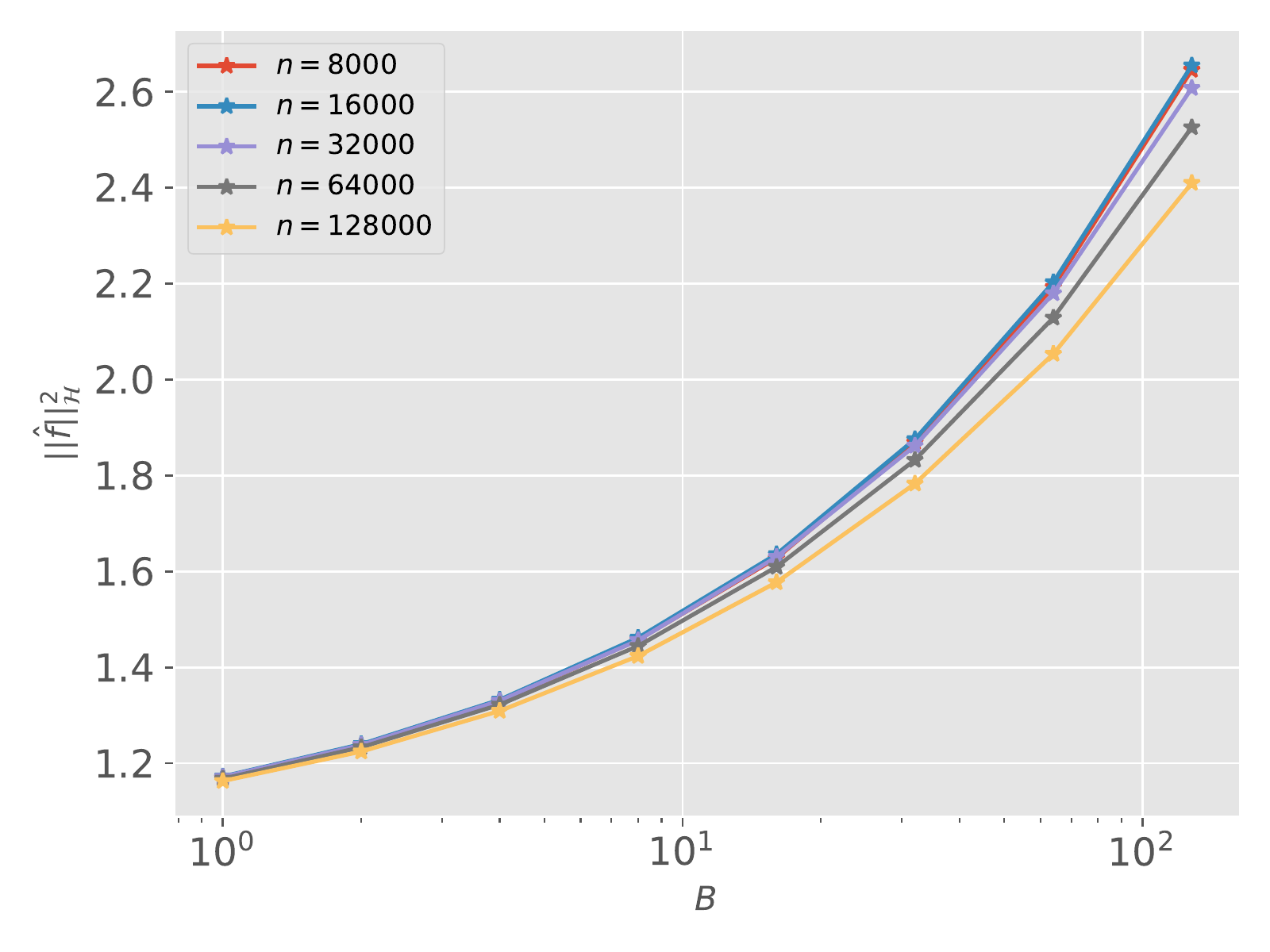}
\caption{Results based on computing the regularized KRR estimate for
  the ``bad'' problems, indexed by the pair $(\numobs, \LRBound)$,
  that underlie the proof of Theorem~\ref{ThmConstrainedKRR}.  Each
  curve shows the squared Hilbert norm of the regularized KRR estimate
  $\|\fhatkrr\|_\Hilbert^2$, computed with $\lambda = 4^{2/3}
  \numobs^{-2/3} \LRBound^{-1/3}$, versus the likelihood ratio bound
  $\LRBound$.  Each curve corresponds to a different choice of sample
  size $\numobs$ as indicated in the legend. }
\label{FigHilbertNorm}
\end{center}
\end{figure}
In Figure~\ref{FigHilbertNorm}, for each fixed $\numobs$, we plot the
squared Hilbert norm $\|\fhatkrr\|_\Hilbert^2$ of the regularized KRR
estimate versus the parameter $\LRBound$.  We vary the choice of
sample size $\numobs \in \{ 8000, 16000, 32000, 64000, 128000\}$, as
indicated in the figure legend.  In all of these curves, we see that
the squared Hilbert norm is increasing as a polynomial function of
$\LRBound$.  This behavior is to be expected, given the sub-optimality
of the constrained KRR estimate with a fixed radius.


\section{Unbounded likelihood ratios}
\label{SecLRMoment}

Thus far, our analysis imposed the $\LRBound$-bound~\eqref{EqnLRBound}
on the likelihood ratio.  In practice, however, it is often the case
that the likelihood ratio is unbounded.  As a simple univariate
example, suppose that the target distribution $\Target$ is standard
normal $\mathcal{N}(0,1)$, whereas the source distribution $\Source$
takes the form $\mathcal{N}(0, 0.9)$.  It is easy to see that the
likelihood ratio $\likeratio(x)$ tends to $\infty$ as $|x| \to
\infty$.  On the other hand, the second moment of the likelihood ratio
under $\Source$ remains bounded, so that
$\chi^2$-condition~\eqref{EqnLRMoment} applies.  

The key to the success of the \emph{unweighted} KRR estimator~\eqref{EqnUnweightedKRR} 
in the bounded likelihood ratio case is the nice relationship between the covariance 
$\PopCov_{\Source} \coloneqq
\Exp_{X \sim \Source} [\phi(X) \phi(X)^\top ]$
of the source distribution  and the covariance $\IdMat$ of the target distribution, 
namely $ \PopCov_{\Source} \succeq \tfrac{1}{\LRBound} \IdMat$. 
In contrast, such a nice relationship (with 
$\LRBound$ replaced by $\VarBound$) does
not appear to hold with unbounded likelihood ratios. It
is therefore natural to consider the likelihood-reweighted
estimate~\eqref{EqnIntroReweightedKRR}, as previously introduced in
Section~\ref{SecUnwRew}, that ensures the nice identity 
$\Exp_{X \sim \Source} [ \likeratio(X) \phi(X) \phi(X)^\top ] = \IdMat$.  
In contrast to the unweighted KRR estimator,
it requires knowledge of the likelihood ratio, but we will see
that---when combined with a suitable form of truncation---it achieves
minimax-optimal rates (up to logarithmic factors) over the much larger
classes of $\chi^2$-bounded source-target pairs.

As noted before, one concern with likelihood-reweighted estimators is
that they can lead to substantial inflation of the variance, in
particular due to the multiplication by the potentially unbounded
quantity $\likeratio(x)$.  For this reason, it is natural to consider
truncation: more precisely, for a given $\TruncLevel > 0$, we define
the \emph{truncated likelihood ratio}
\begin{align}
  \TruncLike{x} & \coloneqq \; \begin{cases}
    \likeratio(x), & \mbox{if $\likeratio(x) \leq \TruncLevel$}, \\
    \TruncLevel, & \mbox{otherwise.}
  \end{cases}
\end{align}
We then consider the family of estimators
\begin{align}
\label{EqnReweightedKRR}
\fhatrw \coloneqq \arg \min_{f \in \Hilbert} \Big \{ \frac{1}{ \numobs
} \sum_{i=1}^\numobs \TruncLike{x_i}(f(x_i) - y_i)^2 + \lambda
\|f\|_{\Hilbert}^2 \Big \},
\end{align}
where $\lambda > 0$, along with the truncation level $\TruncLevel$, are parameters
to be specified.

We analyze the behavior of this estimator for kernels whose
eigenfunctions are $1$-uniformly bounded in sup-norm, meaning that
\begin{align}
\label{EqnEigenBounded}
\|\phi_j\|_\infty \coloneqq \sup_{x \in \Xspace} |\phi_{j}(x)| & \leq
1 \quad \mbox{for all $j = 1, 2, \ldots$.}
\end{align}
Our choice of the constant $1$ is for notational simplicity.  Although
there exist kernels whose eigenfunctions are not uniformly bounded,
there are many kernels for which this condition does hold.  Whenever
the domain $\Xspace$ is compact and the eigenfunctions are continuous,
this condition will hold.  Another class of examples is given by
convolutional kernels (i.e., kernels of the form $\KerFun(x,z) = \Psi(x - z)$
for some $\Psi: \Xspace \rightarrow \real$), which have sinusoids as
their eigenfunctions, and thus satisfy this condition.

\newcommand{\KerComplex}{\ensuremath{\Psi}}

Our theorem on the truncated-reweighted KRR
estimate~\eqref{EqnReweightedKRR} involves the kernel complexity
function $\KerComplex(\delta, \mu) \coloneqq \sum_{j=1}^\infty \min
\{\delta^2, \mu_j  \radius^2 \}$, and works for any solution $\rcrit >
0 $ to the inequality $\Mfun(\delta) \leq  \delta^2  / 2$, where 
\begin{align}
\label{eq:M-function}  
\Mfun(\delta) & \coloneqq \Ccon \sqrt{ \tfrac{\sigma^2 \VarBound
    \log^3(\numobs)}{ \numobs} \KerComplex(\delta, \mu) }.
\end{align}
Here $\Ccon$ is a universal constant, whose value is specified via the
proof.

Below, we present the performance guarantee of $\fhatrw$ in the 
large noise regime (i.e., when $\sigma^2 \geq \kappa^2 \radius^2$) to simplify 
the statement. Theoretical 
guarantees for all ranges of $\sigma^2$ can be found in Appendix~\ref{SecGeneralNoise}. 

\begin{theorem}
  \label{ThmReweightedKRR}
Consider a kernel with sup-norm bounded
eigenfunctions~\eqref{EqnEigenBounded}, and a source-target pair with
$\Exs_\Source[\likeratio^2(X)] \leq V^2$.  
Further assume that the noise level obeys
$\sigma^2 \geq \kappa^2 \radius^2$. Then the estimate $\fhatrw$
with truncation $\TruncLevel = \sqrt{\numobs \VarBound}$ and
regularization $\lambda \radius^2 \geq \rcrit^2 / 3$ satisfies the bound
\begin{align}
\|\fhatrw - \TrueFun \|_{\Target}^2 \leq  \rcrit^2 + 3 \lambda \radius^2
\end{align}
with probability at least $1 - c \: \numobs^{-10}$.  Here, we recall that $\rcrit >
0 $ is 
any solution  to the inequality $\Mfun(\delta) \leq  \delta^2  / 2$, where 
\begin{align*}  
\Mfun(\delta) & = \Ccon \sqrt{ \tfrac{\sigma^2 \VarBound
    \log^3(\numobs)}{ \numobs} \KerComplex(\delta, \mu) }.
\end{align*}

\end{theorem}
\noindent  See Section~\ref{SecProofThmReweightedKRR} 
for the proof of this claim.
In Appendix~\ref{sec:cor-KRR-unbounded-upper}, we also present a corollary
which provides a corresponding expectation bound for the reweighted estimator $\fhatrw$
for such $\VarBound$-bounded covariate shifts. 

\medskip
To appreciate the connection to our previous analysis, in the proof of
Corollary~\ref{CorReweightedKRR} below, we show that for any regular
sequence of eigenvalues and $\radius = 1$, we have
\begin{align}
\label{EqnComplexBound}  
  \KerComplex(\delta, \mu) \leq c' \; d(\delta) \delta^2
\end{align}
for some universal constant $c'$.  Moreover, the condition
$\Mfun(\delta) \leq \delta^2 / 2$ can be verified by checking the
inequality
\begin{align}
\label{EqnTwoInequal}  
\sqrt{\frac{\sigma^2 V^2 \log^3(\numobs)}{\numobs} d(\delta)} \leq c_1
\delta.
\end{align}
This further allows us to obtain the rates of estimation over specific 
kernel classes.

\begin{corollary}
  \label{CorReweightedKRR}
  Consider kernels with sup-norm bounded eigenfunctions~\eqref{EqnEigenBounded}.
  \begin{enumerate}
  \item[(a)] For a kernel with rank $\rank$, the truncated-reweighted
    estimator with $\lambda = \unicon \frac{ \rank \VarBound
  \log^3(\numobs) \sigma^2}{\numobs}$ achieves
    \begin{align}
      \label{EqnFiniteReweighted}
      \|\fhatrw - \TrueFun \|_{\Target}^2 \leq \unicon' \frac{ \rank \VarBound
  \log^3(\numobs) \sigma^2}{\numobs}
    \end{align}
    with high probability.
  \item[(b)] For a kernel with $\alpha$-polynomial eigenvalues, we
    have with high probability 
    \begin{align}
      \label{EqnAlphaReweighted}
      \|\fhatrw - \TrueFun \|_{\Target}^2 \leq \unicon' 
      \Big( \frac{ \VarBound \log^{3} (\numobs) }{ \numobs
}\sigma^2 \Big )^{ \frac{2\alpha}{ 2\alpha +1 } },
    \end{align}
    provided that $\lambda = \unicon \big( \frac{ \VarBound \log^{3}( \numobs )}{ \numobs
}\sigma^2 \big )^{ \frac{2\alpha}{ 2\alpha +1 } }$.
  \end{enumerate}
\end{corollary}

\begin{proof}
  We begin by verifying the claim~\eqref{EqnComplexBound}.  Recalling
  the definition of $d(\delta)$ as the smallest integer for which
  $\mu_j \leq \delta^2$, we can write
  \begin{align*}
\KerComplex(\delta, \mu) & = \sum_{j=1}^{d(\delta)} \min \{ \delta^2,
\mu_j  \} + \sum_{j = d(\delta) +1 }^\infty \min \{\delta^2, \mu_j \}
\; \leq \; d(\delta) \delta^2 + c d(\delta) \delta^2
  \end{align*}
where the bound on the second sum follows from the regularity
condition.  This completes the proof of the
bound~\eqref{EqnComplexBound}. 

Given our bound~\eqref{EqnComplexBound}, it is straightforward to 
verify the claim~\eqref{EqnTwoInequal}.

We now prove the bounds~\eqref{EqnFiniteReweighted} and~\eqref{EqnAlphaReweighted}.
For the finite rank case, we have $\KerComplex(\delta, \mu) \leq \rank \delta^2$, which implies $\rcrit^2 \leq c \frac{\rank \VarBound
  \log^3(\numobs) \sigma^2}{\numobs}$ for some universal constant $c$. 
 Apply Theorem~\ref{ThmReweightedKRR} to obtain the desired rate. 
Now we move on to the kernel with $\alpha$-polynomial eigenvalues. 
We know from the proof of Corollary~\ref{CorKRR} that $d(\delta) \leq 
c (1/\delta)^{1/\alpha}$, and hence $\KerComplex(\delta, \mu) 
\leq c' \delta^{2 - 1/\alpha}$. This implies  
$
\rcrit^2 \leq c \Big( \frac{ \VarBound \log^{3} \numobs }{ \numobs
}\sigma^2 \Big )^{ \frac{2\alpha}{ 2\alpha +1 } },
$
which together with Theorem~\ref{ThmReweightedKRR} yields the claim. 
\end{proof}
Corollary~\ref{CorReweightedKRR} showcases that the reweighted KRR 
estimator is minimax optimal (up to log factors) over this more general 
$\chi^2$-bounded family. This can be seen from the lower bound established in 
Theorem~\ref{ThmLowerBound} and the fact that the $\chi^2$-bounded family 
is a larger family compared to the uniformly bounded family.

\section{Proofs}
\label{SecProofs}

In this section, we provide the proofs of our two sets of upper bounds
on different estimators.  Section~\ref{sec:proof-KRR} is devoted to
the proof of Theorem~\ref{ThmKRR} on upper bounds on unweighted KRR
for $\LRBound$-bounded likelihood ratios, whereas
Section~\ref{SecProofThmReweightedKRR} is devoted to the proof of
Theorem~\ref{ThmReweightedKRR} on the performance of
LR-reweighted KRR with truncation.


\subsection{Proof of Theorem~\ref{ThmKRR}}
\label{sec:proof-KRR}

Define the empirical covariance operator\footnote{In this proof, all
the operators are defined with respect to the space
$\ell^2(\mathbb{N})$.}
\begin{align}
\EmpCov_{\Source} & \coloneqq \frac{1}{\numobs} \sum_{i=1}^{\numobs}
\phi(x_i) \phi(x_i)^\top,
\end{align}
the population covariance operator $\PopCov_{\Source} \coloneqq
\Exp_{X \sim \Source} [\phi(X) \phi(X)^\top ]$, and the diagonal
operator $\Mmat \coloneqq \diag( \{\mu_j\}_{j \geq 1} )$.

Before we embark on the proof, we single out two important properties
regarding $\PopCov_{\Source}$ and $\EmpCov_{\Source}$ that will be
useful in later proofs.  For a given $\lambda > 0$, we define the
event
\begin{align}
\label{eq:good-event-bounded}
\Event(\lambda) \coloneqq \Big \{ \Mmat^{1/2} \EmpCov_{\Source}
\Mmat^{1/2} + \lambda \IdMat \succeq \frac{1}{2} \big ( \Mmat^{1/2}
\PopCov_{\Source} \Mmat^{1/2} + \lambda \IdMat \big ) \Big \},
\end{align}
where $\IdMat$ is the identity operator on $\ell^2(\mathbb{N})$.

\begin{lemma}
\label{LemBasicCov}
For any $\LRBound$-bounded source-target pair~\eqref{EqnLRBound}, we
have the deterministic lower bound
\begin{subequations}
  \begin{align}
\label{EqnPopCovLower}    
    \PopCov_{\Source} & \succeq \tfrac{1}{\LRBound} \IdMat.
  \end{align}
If, in addition, the kernel is $\kappa$-uniformly
bounded~\eqref{EqnKerBound}, then whenever $\numobs \lambda \geq 10
\kappa^2$, the event $\Event(\lambda)$ defined in 
equation~\eqref{eq:good-event-bounded} satisfies
\begin{align}
\Prob[\Event(\lambda)] & \geq 1 - 28 \; \tfrac{\kappa^2}{ \lambda } e^{-\frac{
    \numobs \lambda }{16 \kappa^2}}.
\end{align}
\end{subequations}
\end{lemma}
\noindent See Section~\ref{SecProofLemBasicCov} for the proof
of this claim. \\

Equipped with Lemma~\ref{LemBasicCov}, we now proceed to the proof of
the theorem.  In terms of the basis $\{\phi_j\}_{j \geq 1}$, the KRR
estimate has the expansion $\fhatkrr = \SUMJ \thetakrr_j \phi_j$,
where $\thetakrr = \{\thetakrr_j\}_{j \geq 1}$ is a sequence of
coefficients in $\ell^2(\Natural)$.  By the optimality conditions
for the KRR problem, we have
\begin{align}\label{eqn:opt-conditions-KRR}
\thetahat - \thetastar & = - \lambda ( \EmpCov_{\Source} + \lambda
\Mmat^{-1} )^{-1} \Mmat^{-1} \thetastar + ( \EmpCov_{\Source} +
\lambda \Mmat^{-1} )^{-1} \big( \frac{1}{\numobs}
\sum_{i=1}^{\numobs} \noise_i \phi (x_i) \big).
\end{align}
By the triangle inequality, we have the upper bound $\| \thetahat -
\thetastar \|_2^2 \leq 2 \big( \Term_1 + \Term_2 \big)$, where
\begin{align*}
\Term_1 \coloneqq \| \lambda ( \EmpCov_{\Source} + \lambda \Mmat^{-1}
)^{-1} \Mmat^{-1} \thetastar \|_2^2,  \quad  \mbox{and}  \quad 
\Term_2 \coloneqq \| ( \EmpCov_{\Source} + \lambda \Mmat^{-1} )^{-1}
\big( \frac{1}{\numobs} \sum_{i=1}^{\numobs} \noise_i \phi (x_i)
\big) \|_2^2.
\end{align*}
In terms of this decomposition, it suffices to establish that the
following bounds
\begin{align}
\label{EqnKRRTermBounds}
\Term_1 \stackrel{(a)}{\leq} 2  \lambda \LRBound \radius^2, \quad \mbox{and}
\quad
\Term_2 \stackrel{(b)}{\leq} \frac{40 ( \log \numobs ) \sigma^2 }{
  \numobs } \sum_{j=1}^{\infty} \frac{\mu_j }{ \mu_j / \LRBound +
  \lambda},
\end{align}
hold with probability at least $1 - 28 \; \tfrac{\kappa^2}{ \lambda } e^{-\frac{
    \numobs \lambda }{16 \kappa^2}} - n^{-10}$.


\subsubsection{Proof of the bound~\eqref{EqnKRRTermBounds}(a)}

We establish that this bound holds conditionally on the event
$\Event(\lambda)$.  Following some algebraic manipulations, we have
\begin{align*}
\Term_1 & = \lambda^2 \| \Mmat^{1/2} ( \Mmat^{1/2} \EmpCov_{\Source}
\Mmat^{1/2} + \lambda \IdMat )^{-1} \Mmat^{-1/2} \thetastar \|_2^2 \\
& \stackrel{(i)}{\leq} \lambda^2 \radius^2 \opnorm{\Mmat^{1/2} ( \Mmat^{1/2}
  \EmpCov_{\Source} \Mmat^{1/2} + \lambda \IdMat )^{-1}}^2 \\
& \stackrel{(ii)}{\leq} \lambda \radius^2 \opnorm{\Mmat^{1/2} ( \Mmat^{1/2}
  \EmpCov_{\Source} \Mmat^{1/2} + \lambda \IdMat )^{-1/2} }^2 \\
& \stackrel{(iii)}{\leq} 2 \lambda \radius^2 \opnorm{ \Mmat^{1/2} ( \Mmat^{1/2}
  \PopCov_{\Source} \Mmat^{1/2} + \lambda \IdMat )^{-1} \Mmat^{1/2} }.
\end{align*}
Here inequality (i) follows from the fact that
 $\| \Mmat^{-1/2}
\thetastar \|_2 = \radius$; the second step (ii) uses the fact that
$\Mmat^{1/2} \EmpCov_{\Source} \Mmat^{1/2} + \lambda \IdMat \succeq
\lambda \IdMat$, and step (iii) follows from the fact that we are
conditioning on the event $\Event(\lambda)$.

Lemma~\ref{LemBasicCov} also guarantees that $\PopCov_{\Source}
\succeq \frac{1}{\LRBound} \IdMat$, whence
\begin{align*}
\Term_1 \leq 2 \lambda \radius^2 \opnorm{ \Mmat^{1/2} ( \frac{1}{\LRBound} \Mmat
  + \lambda \IdMat )^{-1} \Mmat^{1/2} } = 2 \lambda \cdot \max_{j \geq
  1} \left \{ \frac{\mu_j}{ \tfrac{\mu_j}{\LRBound} + \lambda} \right
\} \leq 2  \lambda \LRBound \radius^2.
\end{align*}
This establishes the claim~\eqref{EqnKRRTermBounds}(a).


\subsubsection{Proof of the bound~\eqref{EqnKRRTermBounds}(b)}

Define the random vector $W \coloneqq (\EmpCov_{\Source} + \lambda
\Mmat^{-1} )^{-1} \big( \frac{1}{\numobs} \sum_{i=1}^{\numobs}
\noise_i \phi (x_i) \big)$.  Conditioned on the covariates $\{ x_i
\}_{i=1}^\numobs$, $W$ is a zero-mean sub-Gaussian random variable with covariance 
operator
\begin{align*}
\LamMat \coloneqq \frac{\sigma^2}{ \numobs } ( \EmpCov_{\Source} +
\lambda \Mmat^{-1} )^{-1} \EmpCov_{\Source} ( \EmpCov_{\Source} +
\lambda \Mmat^{-1} )^{-1}.
\end{align*}
Consequently, by the Hanson-Wright inequality in the RKHS (cf.~Theorem 2.6 in 
the paper~\cite{chen2021hanson}), we
have
\begin{align}
  \label{EqnTermTwoConcentrate}
\Prob \big[ \Term_2 \geq 20 (\log \numobs) \,\trace(\LamMat) \mid \{x_i\}_{i=1}^\numobs
  \big] & \leq \frac{1}{\numobs^{10}},
\end{align}
where the probability is taken over the noise variables.

It remains to upper bound the trace.  We have $\trace(\LamMat) =
\trace \big ( \frac{\sigma^2}{ \numobs } ( \EmpCov_{\Source} + \lambda
\Mmat^{-1} )^{-1} \EmpCov_{\Source} ( \EmpCov_{\Source} + \lambda
\Mmat^{-1} )^{-1} \big )$, so that
\begin{align*}
\trace ( \LamMat ) & \leq \trace \Big ( \frac{\sigma^2}{ \numobs } (
\EmpCov_{\Source} + \lambda \Mmat^{-1} )^{-1} ( \EmpCov_{\Source} +
\lambda \Mmat^{-1} ) ( \EmpCov_{\Source} + \lambda \Mmat^{-1} )^{-1}
\Big ) \\
& = \trace \Big ( \frac{\sigma^2}{ \numobs } ( \EmpCov_{\Source} +
\lambda \Mmat^{-1} )^{-1} \Big ) = \trace \Big ( \frac{\sigma^2}{
  \numobs } ( \Mmat^{1/2} ( \Mmat^{1/2} \EmpCov_{\Source} \Mmat^{1/2}
+ \lambda \IdMat )^{-1} \Mmat^{1/2} \Big ).
\end{align*}
Under the event $\Event(\lambda)$, we have $\Mmat^{1/2}
\EmpCov_{\Source} \Mmat^{1/2} + \lambda \IdMat \succeq \frac{1}{2}
\left ( \Mmat^{1/2} \PopCov_{\Source} \Mmat^{1/2} + \lambda \IdMat
\right )$, which implies
\begin{align*}
\trace (\LamMat) & \leq 2 \frac{\sigma^2}{ \numobs } \trace \Big(
\Mmat^{1/2} (\Mmat^{1/2} \PopCov_{\Source} \Mmat^{1/2} + \lambda
\IdMat )^{-1} \Mmat^{1/2} \Big) \\
& \stackrel{(i)}{\leq} 2 \frac{\sigma^2}{ \numobs } \trace \Big(
\Mmat^{1/2} ( \frac{1}{\LRBound} \Mmat + \lambda \IdMat )^{-1}
\Mmat^{1/2} \Big) \\
& \stackrel{(ii)}{=} 2 \frac{\sigma^2}{ \numobs } \sum_{j=1}^{\infty}
\frac{\mu_j }{ \tfrac{\mu_j}{\LRBound} + \lambda }.
\end{align*}
Here step (i) follows since $\PopCov_{\Source} \succeq
\frac{1}{\LRBound} \IdMat$, and step (ii) follows from a direct
calculation.  Substituting this upper bound on the trace into the tail
bound~\eqref{EqnTermTwoConcentrate} yields the claimed
bound~\eqref{EqnKRRTermBounds}(b).


\subsubsection{Proof of Lemma~\ref{LemBasicCov}}
\label{SecProofLemBasicCov}
We begin with the proof of the lower bound~\eqref{EqnPopCovLower}.
Since $\{ \phi_j \}_{j \geq 1}$ is an orthonormal basis of
$L^2(\Target)$, we have
\begin{align*}
\Exp_{X \sim \Target}\big[ \phi(X) \phi(X)^\top \big] = \Exp_{X \sim
  \Source} \big[ \likeratio(X) \phi(X) \phi(X)^\top \big] = \IdMat.
\end{align*}
Thus, the $\LRBound$-boundedness of the likelihood
ratio~\eqref{EqnLRBound} implies that
\begin{align*}
\IdMat \preceq \Exp_{X \sim \Source} [ \LRBound \phi(X) \phi(X)^\top ]
= \LRBound \PopCov_{\Source},
\end{align*}
which is equivalent to the claim~\eqref{EqnPopCovLower}.

Next we prove the lower bound~\eqref{eq:good-event-bounded}.
We introduce the shorthand notation
\begin{align*}
\EmpCov_{\lambda} \coloneqq \Mmat^{1/2} \EmpCov_{\Source} \Mmat^{1/2}
+ \lambda \IdMat,  \quad  \mbox{and}  \quad  \PopCov_{\lambda} \coloneqq
\Mmat^{1/2} \PopCov_{\Source} \Mmat^{1/2} + \lambda \IdMat
\end{align*}
along with the matrix $\DelMat \coloneqq \PopCov_{\lambda}^{-1/2} (
\EmpCov_{\lambda} - \PopCov_{\lambda} )
\PopCov_{\lambda}^{-1/2}$. Recalling that $\opnorm{\cdot}$ denotes the
$\ell_2$-operator norm of a matrix, we first observe that $\{
\opnorm{\DelMat} \leq \frac{1}{2} \} \subseteq \Event$.  Consequently,
it suffices to show that $\opnorm{\DelMat} \leq \tfrac{1}{2}$ with
high probability.

The matrix $\DelMat$ can be written as the normalized sum $\DelMat =
\frac{1}{\numobs} \sum_{i=1}^\numobs \Zmat_i$, where the random
operators
\begin{align*}
\Zmat_i & \coloneqq \PopCov_{\lambda}^{-1/2} \Mmat^{1/2} \left (
\phi(x_i) \phi(x_i)^\top - \PopCov_{\Source} \right ) \Mmat^{1/2}
\PopCov_{\lambda}^{-1/2}
\end{align*}
are i.i.d. The operator norm of each term can be bounded as
\begin{align}
\opnorm{\Zmat_i} \leq 2 \sup_{x \in \Xspace} \opnorm{
  \PopCov_{\lambda}^{-1/2} \Mmat^{1/2} \phi(x) \phi(x)^\top
  \Mmat^{1/2} \PopCov_{\lambda}^{-1/2} } & = 2 \sup_{x \in \Xspace} \|
\PopCov_{\lambda}^{-1/2} \Mmat^{1/2} \phi(x) \|_2^2 \nonumber \\
& \leq 2 \kappa^2 \opnorm{ \PopCov_{\lambda}^{-1/2} }^2 \leq \frac{2
  \kappa^2 }{\lambda} \eqqcolon L, \label{eq:L-bound}
\end{align}
where the final inequality follows from the assumption that $\sup_{x
  \in \Xspace} \| \Mmat^{1/2} \phi(x) \|_2^2 \leq \kappa^2$, and the
fact that $\PopCov_{\lambda} \succeq \lambda \IdMat$.

On the other hand, the variance of $\Zmat_i$ can be bounded as
\begin{align*}
\Exp [\Zmat_i^2] & \preceq \Exp [ ( \PopCov_{\lambda}^{-1/2}
  \Mmat^{1/2} \phi(X) \phi(X)^\top \Mmat^{1/2}
  \PopCov_{\lambda}^{-1/2} )^2 ] \\
& = \Exp [ \PopCov_{\lambda}^{-1/2} \Mmat^{1/2} \phi(X) \phi(X)^\top
  \Mmat^{1/2} \PopCov_{\lambda}^{-1} \Mmat^{1/2} \phi(X) \phi(X)^\top
  \Mmat^{1/2} \PopCov_{\lambda}^{-1/2} ] \\
& \preceq \Exp \Big[ \PopCov_{\lambda}^{-1/2} \Mmat^{1/2} \phi(X)
  \phi(X)^\top \Mmat^{1/2} \PopCov_{\lambda}^{-1/2} \Big] \cdot
\sup_{x \in \Xspace} \Big \{ \phi(x)^\top \Mmat^{1/2}
\PopCov_{\lambda}^{-1} \Mmat^{1/2} \phi(x) \Big \} \\
& \preceq \frac{  \kappa^2 }{\lambda} \PopCov_{\lambda}^{-1/2}
\Mmat^{1/2} \PopCov_{\Source} \Mmat^{1/2} \PopCov_{\lambda}^{-1/2}
\eqqcolon \Vmat,
\end{align*}
where the last inequality follows by applying the bound~\eqref{eq:L-bound}
 on $\sup_{x
  \in \Xspace} \| \PopCov_{\lambda}^{-1/2} \Mmat^{1/2} \phi(x)
\|_2^2$.  

Suppose that we can show that
\begin{subequations}
\label{subeq:bernstein}
\begin{align}
\trace(\Vmat) &\leq \frac{ \kappa^2}{\lambda} \cdot
\frac{\kappa^2}{\lambda}; \label{eq:trace-bound}\\ 
\opnorm{\Vmat} &\leq \frac{
  \kappa^2}{\lambda}. \label{eq:op-norm-bound}
\end{align}
\end{subequations} 
We can then apply a dimension-free matrix Bernstein inequality (see 
Lemma~\ref{lemma:bernstein-operator}) with $t=1/2$ to obtain the tail bound
\begin{align*}
\Prob \big[ \opnorm{\DelMat} \geq \tfrac{1}{2} \big] & \leq 28 
\frac{\kappa^2}{ \lambda } \exp \Big( -\frac{ \numobs \lambda }{
  16 \kappa^2 } \Big),
\end{align*}
as long as $\numobs \lambda \geq 10 \kappa^2$.  Thus, the only
remaining detail is to prove the bounds~\eqref{eq:trace-bound}
and~\eqref{eq:op-norm-bound}.

\paragraph{Proof of the bound~\eqref{eq:trace-bound}:}

Using the definition of $\Vmat$, we have
\begin{align*}
\trace(\Vmat) = \frac{ \kappa^2}{\lambda}
\trace\Big(\PopCov_{\lambda}^{-1/2} \Mmat^{1/2} \PopCov_{\Source}
\Mmat^{1/2} \PopCov_{\lambda}^{-1/2} \Big) & = \frac{
  \kappa^2}{\lambda} \Exp_{\Source} [ \trace \big(
  \PopCov_{\lambda}^{-1/2} \Mmat^{1/2} \phi(x) \phi(x)^\top
  \Mmat^{1/2} \PopCov_{\lambda}^{-1/2} \big ) ] \\
& \leq \frac{ \kappa^2}{\lambda} \cdot
\frac{\kappa^2}{\lambda}.
\end{align*}
Here we have again applied the bound $\sup_{x \in \Xspace} \|
\PopCov_{\lambda}^{-1/2} \Mmat^{1/2} \phi(x) \|_2^2 \leq \kappa^2 /
\lambda$.

\paragraph{Proof of the bound~\eqref{eq:op-norm-bound}:}

Recalling the definition of $\PopCov_{\lambda}$, we see that $\opnorm{
  \PopCov_{\lambda}^{-1/2} \Mmat^{1/2} \PopCov_{\Source} \Mmat^{1/2}
  \PopCov_{\lambda}^{-1/2} } \leq 1$, and hence
\begin{align*}
\opnorm{\Vmat} = \frac{ \kappa^2 }{\lambda} \opnorm{
  \PopCov_{\lambda}^{-1/2} \Mmat^{1/2} \PopCov_{\Source} \Mmat^{1/2}
  \PopCov_{\lambda}^{-1/2} } \leq \frac{ \kappa^2 }{\lambda},
\end{align*}
which is the claimed upper bound on $\opnorm{\Vmat}$.



\subsection{Proof of Theorem~\ref{ThmReweightedKRR}}
\label{SecProofThmReweightedKRR}

We now turn to the proof of our guarantee on the truncated
LR-reweighted estimator.  At the core of the proof is a uniform
concentration result, one that holds within a local ball
\begin{align*}
        \GoodFunc(r) \coloneqq \{ f \in \Hilbert \mid \| f - \TrueFun\|_{\Target} 
        \leq r, \text{ and }  \|f - \TrueFun\|_{\Hilbert} \leq 3 \radius \}
\end{align*}
around the
true regression function~$\TrueFun$.
\begin{lemma}
\label{lemma:key}
Fixing any $r > 0$, we have
\begin{align}
\label{eq:key-event}
\sup_{ \substack{g \in \GoodFunc(r)}} \Big \{ \| g - \TrueFun\|_{\Target}^2 + \frac{1}{\numobs}
\sum_{i=1}^{\numobs} \TruncLike{x_i} \big[\big( \TrueFun(x_i) - y_{i}
  \big)^2 - \big( g(x_i) - y_{i} \big)^2 \big] \Big \} \leq \Mfun(r)
\end{align}
with probability at least $1 - c \numobs^{-10}$.
\end{lemma}
\noindent See Section~\ref{SecProofLemmaKey} for the proof of this
lemma.

\medskip

\newcommand{\regCrit}{\delta_\lambda}
Taking this lemma as given, we now complete the proof of the theorem. 
Define the regularized radius $\regCrit \coloneqq \sqrt{\rcrit^2 + 3 \lambda \radius^2}$, 
and denote by $\Event(\regCrit)$ the ``good'' event that the
relation~\eqref{eq:key-event} holds at radius $\regCrit$. 
We immediately point out an important property of the regularized radius 
$\regCrit$, namely $\Mfun(\regCrit) \leq (1/2) \cdot \regCrit^2$. To see this, 
note that $r \mapsto \Mfun(r) / r$ is non-increasing in $r$, and hence 
\begin{align*}
\frac{\Mfun(\regCrit)}{\regCrit} \leq \frac{\Mfun(\rcrit)}{\rcrit} \leq \frac{1}{2} \rcrit 
\leq \frac{1}{2} \regCrit.
\end{align*}

Suppose that
conditioned on $\Event(\regCrit)$, the following inequality holds
\begin{align}
\label{eq:sufficient-condition}  
\inf_{f \in \Hilbert, f \notin \GoodFunc(\regCrit) }
\frac{1}{\numobs} \sum_{i=1}^{\numobs} \TruncLike{x_i} \Big \{ \big(
f(x_i) - y_{i} \big)^2 - \big( \TrueFun(x_i) - y_{i} \big)^2 \Big \}
+ \lambda  \|f\|_{\Hilbert}^2 - \lambda  \|\TrueFun\|_{\Hilbert}^2> 0.
\end{align}
It then follows that that $\| \fhat -\TrueFun \|_{\Target} \leq
\regCrit$, as desired.  Consequently, the remainder of our proof is
devoted to establishing that
inequality~\eqref{eq:sufficient-condition} holds conditioned
on $\Event(\regCrit)$.

Given any function $f \in \Hilbert$ and $f \notin \GoodFunc(\regCrit)$, 
there exists an $\alpha \geq 1$ such that $\ftil \coloneqq
\TrueFun + \tfrac{1}{\alpha} (f -
\TrueFun)$ lies in the set $\Hilbert$, and more importantly $\ftil$ lies 
on the boundary of $\GoodFunc(\regCrit)$. 
This follows from the convexity of the two sets $\Hilbert$ and 
$\GoodFunc(\regCrit)$.  Since $\ftil$ is a
convex combination of $f$ and $\TrueFun$, Jensen's inequality implies
that
\begin{align*}
&\TruncLike{x_i} \Big \{ \big( \ftil( x_i ) - y_{i} \big)^2 - \big(
\TrueFun(x_i) - y_{i} \big)^2 \Big \} + \lambda  \|\ftil\|_{\Hilbert}^2 - \lambda  \|\TrueFun\|_{\Hilbert}^2 \\
&\quad  \leq \tfrac{1}{\alpha} \left \{ \TruncLike{x_i}  \Big \{ \big( f(x_i) -
y_{i}\big)^2 - \big( \TrueFun(x_i) - y_{i} \big)^2 \Big \} 
+ \lambda  \|f\|_{\Hilbert}^2 - \lambda  \|\TrueFun\|_{\Hilbert}^2 \right \}.
\end{align*}
Consequently, in order to establish the
claim~\eqref{eq:sufficient-condition}, it suffices to prove that the
quantity
\begin{align*}
\Term \coloneqq \frac{1}{\numobs} \; \sum_{i=1}^{\numobs}
\TruncLike{x_i} \big \{ \big( \TrueFun(x_i) - y_{i} \big)^2 - \big(
\ftil(x_i) - y_{i} \big)^2 \big \} +  \lambda  \|\TrueFun\|_{\Hilbert}^2 -\lambda  \|\ftil\|_{\Hilbert}^2
\end{align*}
is negative. Since $\ftil$ lies on the boundary of $\GoodFunc(\regCrit)$, 
we can split the proof into two cases: (1) $\|\ftil - \TrueFun\|_\Target = \regCrit$, 
while $\|\ftil - \TrueFun\|_\Hilbert \leq 3 \radius$, and (2) $\|\ftil - \TrueFun\|_\Target 
\leq \regCrit$, while $\|\ftil - \TrueFun\|_\Hilbert = 3 \radius$.

\paragraph{Case 1: }$\|\ftil - \TrueFun\|_\Target = \regCrit$, 
while $\|\ftil - \TrueFun\|_\Hilbert \leq 3 \radius$. 
By adding and subtracting terms, we have
\begin{align*}
\Term & = \Biggr[ \frac{1}{\numobs} \sum_{i=1}^{\numobs}
  \TruncLike{x_i} \Big \{ \big( \TrueFun(x_i) - y_{i} \big)^2 -\big(
  \ftil( x_i ) - y_{i} \big)^2 \Big \} + \| \ftil - \TrueFun
  \|_{\Target}^2 \Biggr] -\| \ftil -\TrueFun \|_{\Target}^2 +  \lambda  \|\TrueFun\|_{\Hilbert}^2 -\lambda  \|\ftil\|_{\Hilbert}^2 \\
 & \stackrel{(i)}{\leq} \Mfun(\regCrit) - \regCrit^2 + \lambda \radius^2 
 \stackrel{(ii)}{<} - \frac{1}{2} \regCrit^2 + \lambda \radius^2 \stackrel{(iii)}{<} 0,
\end{align*}
where step (i) follows from conditioning on the event $\Event(\regCrit)$, 
the equality $\|\ftil - \TrueFun \|_{\Target}^2 = \regCrit^2$, and 
non-positivity of $\lambda \|\ftil\|_\Hilbert^2$;
 step (ii) follows from the property $\Mfun(\regCrit) \leq (1/2) \cdot \regCrit^2 $ and 
 step (iii) uses the definitions of $\regCrit$ and~$\lambda$.  
 
\paragraph{Case 2: }$\|\ftil - \TrueFun\|_\Target \leq \regCrit$, 
while $\|\ftil - \TrueFun\|_\Hilbert = 3 \radius$. 
By the same addition and subtraction as above, we have 
\begin{align*}
\Term & = \Biggr[ \frac{1}{\numobs} \sum_{i=1}^{\numobs}
  \TruncLike{x_i} \Big \{ \big( \TrueFun(x_i) - y_{i} \big)^2 -\big(
  \ftil( x_i ) - y_{i} \big)^2 \Big \} + \| \ftil - \TrueFun
  \|_{\Target}^2 \Biggr] -\| \ftil -\TrueFun \|_{\Target}^2 +  \lambda  \|\TrueFun\|_{\Hilbert}^2 -\lambda  \|\ftil\|_{\Hilbert}^2 \\
 & \stackrel{(i)}{\leq} \Mfun(\regCrit) + \lambda  \|\TrueFun\|_{\Hilbert}^2 -\lambda  \|\ftil\|_{\Hilbert}^2 \\
 & \stackrel{(ii)}{<} \frac{1}{ 2} \regCrit^2 - 3 \lambda \radius^2,
\end{align*}
Here, step (i) again follows from the conditioning on the event 
$\Event(\regCrit)$ and the assumption that $\|\ftil - \TrueFun\|_\Target \leq \regCrit$. 
Step (ii) relies on the facts that $\Mfun(\regCrit) \leq (1/2) \cdot \regCrit^2 $, $\|\TrueFun\|_\Hilbert = \radius$, and that $\|\ftil\|_\Hilbert \geq 2 \radius$. The latter is a simple consequence of $\|\ftil - \TrueFun\|_\Hilbert = 3 \radius$ and the triangle inequality. Substitute in the definitions of $\regCrit$ and $\lambda$ to see the negativity of~$\Term$.  

Combine the two cases to 
finish the proof of the claim~\eqref{eq:sufficient-condition}.


\subsubsection{Proof of Lemma~\ref{lemma:key}}
\label{SecProofLemmaKey}

Define the shifted function class $\FuncClass^{\star} \coloneqq
\Hilbert - \TrueFun$, along with its $r$-localized version
\begin{align*}
  \FclassStar(r) \coloneqq \big \{ h \in \FclassStar \mid
  \|h\|_\Target \leq r ,\quad \text{and}\quad \|h\|_\Hilbert \leq 3 \radius \big \}.
\end{align*}
We begin by observing that
\begin{align*}
\big( \TrueFun(x_i) - y_{i} \big)^2 - \big( g(x_i) - y_{i} \big)^2 & =
2 \noise_{i} [ g(x_i) - \TrueFun(x_i) ] - \big( g(x_i) - \TrueFun(x_i)
\big)^2,
\end{align*}
which yields the following equivalent formulation of the claim in Lemma~\ref{lemma:key}:
\begin{align}
\label{EqnEquivalent}
\sup_{h \in \FuncClassStar(r)} \Biggr \{ \frac{1}{\numobs}
\sum_{i=1}^{\numobs} \Big[ 2 \noise_{i} \TruncLike{x_i} h(x_i) +
  \|h\|_{\Target}^2 - \TruncLike{x_i} h^2(x_i) \Big] \Biggr \} & \leq
\Mfun(r).
\end{align}
By the triangle inequality, it suffices to show that $\Term_1 +
\Term_2 \leq \Mfun(r)$, where
\begin{align*}
\Term_1 \coloneqq \sup_{h \in \FuncClassStar(r)} \Big|
\frac{2}{\numobs} \sum_{i=1}^{\numobs} \noise_{i}\TruncLike{x_i}
h(x_i) \Big|,  \quad  \mbox{and}  \quad  \Term_2 \coloneqq \sup_{h \in
  \FuncClassStar(r)} \Big| \frac{1}{\numobs} \sum_{i=1}^{\numobs} \big
\{ \|h\|_{\Target}^2 - \TruncLike{x_i} h^2(x_i) \big \} \Big|.
\end{align*}
More precisely, the core of our proof involves establishing the
following two bounds: 
\begin{subequations}
  \begin{align}
\label{EqnMina}    
\Term_1 & \leq \unicon \sigma \sqrt{\frac{\VarBound \log^3(\numobs)}{\numobs}}
\Bigsqrt{ \sum_{j=1}^{\infty} \min\{ r^2, \mu_{j} \radius^2 \}} \qquad \mbox{with
  probability at least }1 - \numobs^{-10}, \text{ and } \\
\label{EqnKento}
\Term_2 & \leq \unicon \sqrt{ \frac{\VarBound
    \log^3(\numobs)}{\numobs}} \cdot \SUMJ \min\{r^2, \mu_{j} \radius^2 \} \qquad
\mbox{with probability at least }1 - \numobs^{-10}.
 \end{align}
\end{subequations}
In conjunction, these two bounds ensure that
\begin{align}\label{eq:M-original-bound}
\Term_1 + \Term_2 & \leq \unicon \sqrt{ \frac{\VarBound
    \log^3(\numobs)}{\numobs}} \cdot \SUMJ \min\{r^2, \mu_{j} \radius^2 \} +
\unicon \sqrt{\sum_{j=1}^{\infty} \min \{ r^2, \mu_{j} \radius^2  \}
  \frac{\VarBound \log^3(\numobs)}{\numobs} \sigma^2}.
\end{align}
Since the kernel function is $\kappa^2$-bounded, we have 
$\sum_{j=1}^{\infty} \min \{ r^2, \mu_{j} \radius^2 \} \leq \radius^2 \sum_{j=1}^{\infty} \mu_{j} 
\leq \kappa^2 \radius^2$, which together with the assumption $\sigma^2 \geq 
\kappa^2 \radius^2$ implies that $\Term_1 + \Term_2 \leq 2\unicon 
\sqrt{\sum_{j=1}^{\infty} \min \{ r^2, \mu_{j} \radius^2 \}
  \frac{\VarBound \log^3(\numobs)}{\numobs} \sigma^2}$.
Therefore the bound~\eqref{EqnEquivalent} holds.

It remains to prove the bounds~\eqref{EqnMina} and~\eqref{EqnKento}.
The proofs make use of some elementary properties of the localized
function class $\FuncClassStar(r)$, which we collect here.  For any $h
\in \FuncClassStar(r)$, we have
\begin{subequations}
\label{EqnElementary}  
\begin{align}
\label{eq:Linf-bound}  
|h(x)| & \leq \sqrt{10 \SUMJ \min\{ r^2, \mu_{j} \radius^2 \} }, \qquad
\mbox{and} \\
\label{eq:fact}  
\sum_{j=1}^{\infty} \frac{\theta_{j}^2 }{ \min\{r^2,\mu_{j} \radius^2\}} &
\leq 10, \qquad \mbox{where $h = \SUMJ \theta_j \phi_j$.}
\end{align}
\end{subequations}
See Appendix~\ref{AppProofElementary} for the proof of these
elementary claims.

\subsubsection{Proof of inequality~\eqref{EqnKento}}

We begin by analyzing the term $\Term_2$.  By the triangle inequality,
we have the upper bound $\Term_2 \leq \Termtwoa + \Termtwob$, where
\begin{align*}
 \Termtwoa & \coloneqq \sup_{h \in \FuncClassStar(r)} \Big|
 \|h\|_{\Target}^2 - \Exp_{\Source} [\TruncLike{X} h^2(X)] \Big|,
  \quad  \mbox{and} \\
\Termtwob & \coloneqq \sup_{ h \in \FuncClassStar(r)} \Big|
\Exp_{\Source}[\TruncLike{X} h^2(X)] - \frac{1}{\numobs}
\sum_{i=1}^\numobs \TruncLike{x_i}h^2(x_i) \Big \} \Big|.
\end{align*}
Note that $\Termtwoa$ is a deterministic quantity, measuring the bias
induced by truncation, whereas $\Termtwob$ is the supremum of an
empirical process.  We split our proof into analysis of these two
terms.  In particular, we establish the following bounds:
\begin{subequations}
\begin{align}
\label{EqnBoundTermTwoA}
\Termtwoa & \leq \unicon \sqrt{ \frac{ \VarBound}{ \numobs } }\SUMJ
\min\{ r^2, \mu_{j} \radius^2 \},  \quad  \mbox{and} \\
\label{EqnBoundTermTwoB}
\Termtwob & \leq \unicon \sqrt{ \frac{ \VarBound
    \log^2(\numobs)}{\numobs} } \: \sum_{j=1}^{\infty} \min \{ r^2,
\mu_{j} \radius^2 \} \qquad \text{with probability at least } 1 - \numobs^{-10}.
\end{align}
\end{subequations}
Combining these two bounds yields the claim~\eqref{EqnKento}.


\paragraph{Proof of inequality~\eqref{EqnBoundTermTwoA}:}
We begin by proving the claimed upper bound on $\Termtwoa$.  Note that
\begin{align*}
\Termtwoa &\leq \sup_{h \in \FuncClassStar(r)} \Big| \|h\|_{\Target}^2
- \Exp_{\Target} [ \indicator \{\likeratio(X) \leq \TruncLevel \}
  h^2(X) ] \Big|  + \TruncLevel \cdot  \sup_{h \in \FuncClassStar(r)} \Big| \Exp_{\Source} [ \indicator \{\likeratio(X) > \TruncLevel \}
  h^2(X) ] \Big| \\
 & = \sup_{h \in \FuncClassStar(r)} \Exp_{\Target}
\Big[ \indicator \{ \likeratio(X) > \TruncLevel \} h^2(X) \Big] + \TruncLevel \cdot  \sup_{h \in \FuncClassStar(r)} \Big| \Exp_{\Source} [ \indicator \{\likeratio(X) > \TruncLevel \}
  h^2(X) ] \Big| \\
& \leq \Exp_{\Target} \Big[ \indicator \{ \likeratio (X) > \TruncLevel
  \} \Big ] \cdot \sup_{h \in \FuncClassStar(r)} \|h\|_{\infty}^2 + \TruncLevel \cdot \Exp_{\Source} [ \indicator \{\likeratio(X) > \TruncLevel \} ]  \cdot \sup_{h \in \FuncClassStar(r)} \|h\|_{\infty}^2 \\
& \leq \frac{\VarBound}{\TruncLevel} \cdot 10 \SUMJ \min
\{r^2, \mu_{j} \radius^2 \} + \TruncLevel \cdot \frac{\VarBound}{ (\TruncLevel)^2 } \cdot 10 \SUMJ \min
\{r^2, \mu_{j} \radius^2 \},
\end{align*}
where the last step follows from a combination of Markov's inequality, Chebyshev's inequality, 
and the $\ell_{\infty}$-norm bound~\eqref{eq:Linf-bound}. Recalling
that $\TruncLevel = \sqrt{\numobs \VarBound}$, the
bound~\eqref{EqnBoundTermTwoA} follows.


\paragraph{Proof of the bound~\eqref{EqnBoundTermTwoB}:}

We prove the claimed bound on $\Termtwob$ by first bounding its mean
$\Exp[\Termtwob]$, and then providing a high-probability bound on the
deviation $\Termtwob - \Exp[ \Termtwob ]$.

\medskip
\noindent \underline{Bound on the mean:} By a standard symmetrization
argument (see e.g., Chapter 4 in the book~\cite{wainwright2019high}), we have the upper bound
\begin{align*}
\Exp[\Termtwob] \leq \frac{2}{\numobs} \: \Exp \Big[ \sup_{h \in
    \FuncClassStar(r)} \big| \sum_{i=1}^{\numobs}
  \varepsilon_{i}\TruncLike{x_i} h^2(x_i) \big| \Big],
\end{align*}
where $\{\varepsilon_i \}_{i=1}^\numobs$ is an i.i.d.~sequence of
Rademacher variables.  Now observe that
\begin{align*}
  \sup_{h \in \FuncClassStar(r)} \Big| \sum_{i=1}^{\numobs}
  \varepsilon_i \TruncLike{x_i} h^2(x_i) \Big| & \leq \sup_{
    \substack{\htil, h \in \FuncClassStar(r)}} Z(h, \htil),  \quad 
  \mbox{where}  \quad  Z(h, \htil) \coloneqq \Big| \sum_{i=1}^{\numobs}
  \varepsilon_i \TruncLike{x_i} \htil(x_i) h(x_i) \Big|.
\end{align*}
Writing $\htil = \SUMJ \thetatil_j \phi_j$, we have
\begin{align*}
& Z(h, \htil) = \Big| \sum_{j=1}^{\infty} \thetatil_{j} \big \{
\sum_{i=1}^{\numobs} \varepsilon_{i} \TruncLike{x_i} \phi_{j} (x_i)
h(x_i) \big \} \Big| \\
&\quad  = \Big | \sum_{j=1}^{\infty} \tfrac{ \thetatil_{j} }{ \sqrt{
    \min\{r^2,\mu_{j} \radius^2 \}}} \cdot \sqrt{ \min\{ r^2, \mu_{j} \radius^2\} } \big
\{ \sum_{i=1}^{\numobs} \varepsilon_{i} \TruncLike{x_i}\phi_{j} (x_i)
h(x_i) \big\} \Big| \\
&\quad  \leq \sqrt{10} \, \Bigsqrt{ \sum_{j=1}^{\infty} \min\{r^2,\mu_{j} \radius^2 \}
  \cdot \big \{ \sum_{i=1}^{\numobs} \varepsilon_{i} \TruncLike{x_i}
  \phi_{j} (x_i) h(x_i) \big \}^2 },
\end{align*}
where the final step follows by combining the Cauchy--Schwarz
inequality with the bound~\eqref{eq:fact}.  We now repeat the same
argument to upper bound the inner term involving $h$; in particular,
we have
\begin{align*}
\Big \{ \sum_{i=1}^{\numobs} \varepsilon_{i} \TruncLike{x_i}
\phi_{j}(x_i) h(x_i) \Big \}^2 & = \Big\{ \sum_{k=1}^{\infty}
\theta_{k} \big( \sum_{i=1}^{\numobs} \varepsilon_{i}
\TruncLike{x_i} \phi_{j}(x_i)\phi_{k}(x_i) \big) \Big\}^2 \\
& \leq 10 \cdot \sum_{k=1}^{\infty} \left \{ \min\{r^2, \mu_{k} \radius^2\} \Big(
\sum_{i=1}^{\numobs} \varepsilon_{i} \TruncLike{x_i} \phi_{j}(x_i)
\phi_{k}(x_i) \Big)^2 \right\} .
\end{align*}
Putting together the pieces now leads to the upper bound
\begin{align*}
 & \frac{2}{\numobs} \sup_{h \in \FuncClassStar(r)} \Big| \sum_{i=1}^{\numobs}
  \varepsilon_i \TruncLike{x_i} h^2(x_i) \Big|  \leq \frac{2}{\numobs} \sup_{h, \htil \in
  \FuncClassStar(r)} Z(h, \htil) \\
  & \quad  \leq \; \frac{20}{\numobs} \;
\Bigsqrt{ \sum_{j=1}^{\infty} \min\{r^2,\mu_{j} \radius^2 \} \cdot
  \sum_{k=1}^{\infty} \min\{r^2,\mu_{k} \radius^2 \} \big( \sum_{i=1}^{\numobs}
  \varepsilon_{i} \TruncLike{x_i} \phi_{j}(x_i) \phi_{k}(x_i)
  \big)^2}.
\end{align*}
By taking expectations of both sides and applying Jensen's inequality,
we find that
\begin{align}
\label{EqnRock}  
\Exp[\Termtwob] \leq \frac{20}{\numobs} \Bigsqrt{ \sum_{j=1}^{\infty}
  \min \{r^2, \mu_{j} \radius^2 \} \cdot \sum_{k=1}^{\infty} \min\{r^2,\mu_{k} \radius^2 \}
  \Exp_{X, \varepsilon} \big( \sum_{i=1}^{\numobs} \varepsilon_{i}
  \TruncLike{x_i} \phi_{j}(x_i)\phi_{k}(x_i) \big)^2}.
\end{align}
We now observe that 
\begin{align*}
  \Exp_{X,\varepsilon} \Big[ \big(\sum_{i=1}^{\numobs} \varepsilon_{i}
    \TruncLike{x_i} \phi_{j}(x_i) \phi_{k}(x_i) \big)^2 \Big] & =
  \sum_{i=1}^{\numobs} \Exp_{X,\varepsilon} \Big[ \varepsilon_{i}^2 (\TruncLike{x_i})^2
    \phi_{j}^2(x_i)\phi_{k}^2(x_i) \Big] \\
& \leq \sum_{i=1}^{\numobs} \Exp_{X,\varepsilon}[\likeratio^2(x_i)] \;
  \leq \; \numobs \VarBound,
\end{align*}
where we have used the fact that $\|\phi_j\|_\infty \leq 1$ for all $j
\geq 1$, and that $\TruncLike{x_i} \leq \likeratio(x_i)$.  Substituting this upper bound into our earlier
inequality~\eqref{EqnRock} yields
\begin{align}
\label{EqnMeanBoundTermTwoB} 
    \Exp[ \Termtwob ] & \leq 20 \: \sqrt{\frac{\VarBound}{\numobs}}
    \cdot \SUMJ \min \{r^2, \mu_{j} \radius^2 \}.
\end{align}

\medskip

\noindent \underline{Bounding the deviation term:} Recall that for any
$h \in \FuncClass^{\star}$, we have
$\|h\|_{\infty} \leq \sqrt{ 10 \SUMJ \min\{ r^2, \mu_{j} \radius^2 \}}$.
Consequently, we have
\begin{align*}
\sup_{h \in \FuncClassStar(r)} \Big| \Exp_{\Target}[ \indicator \{
  \likeratio(X) \leq \TruncLevel \} h^2(X) ] - \TruncLike{x_i}
h^2(x_i) \Big| & \leq 10 \TruncLevel \SUMJ \min\{ r^2, \mu_{j} \radius^2 \} \\
& = 10
\sqrt{\numobs \VarBound} \SUMJ \min\{ r^2, \mu_{j} \radius^2 \}.
\end{align*}
In addition, we have
\begin{align*}
\sup_{h \in \FuncClassStar(r)} \sum_{i=1}^{\numobs} \Exp \Big[ \big\{
  \Exp_{\Target}[ \indicator \{ \likeratio(X) \leq \TruncLevel \}
    h^2(X) ] - \TruncLike{x_i} h^2(x_i) \big\}^2 \Big] & \leq \sup_{h
  \in \FuncClassStar(r)} \sum_{i=1}^{\numobs} \Exp
\big[(\TruncLike{x_i})^2
  h^{4}(x_i) \big] \\
  & \leq 100 \numobs \VarBound \big( \SUMJ \min\{ r^2, \mu_{j} \radius^2 \}
\big)^2,
\end{align*}
where we have applied the $\ell_{\infty}$-norm
bound~\eqref{eq:Linf-bound} as well as the $\VarBound$-condition on
the likelihood ratio. These two facts together allow us to apply
Talagrand's concentration results (cf.~Lemma~\ref{lemma:talagrand}) 
and obtain
\begin{align}
\label{EqnDeviationBoundTermTwoB}  
\Prob \Big[ \Termtwob \geq \Exp[ \Termtwob ] + \tfrac{t}{\numobs}
  \Big] & \leq \exp \left( - \frac{t^2}{3000 \numobs \VarBound \big(
  \SUMJ \min\{ r^2, \mu_{j} \radius^2 \} \big)^2 + 900 \sqrt{ \numobs \VarBound}
  \SUMJ \min\{ r^2, \mu_{j} \radius^2\} t } \right).
\end{align}

\medskip
\vspace*{0.1in}

\noindent \underline{Completing the proof of the
  bound~\eqref{EqnBoundTermTwoB}:} We now have the ingredients to
complete the proof of the claim~\eqref{EqnBoundTermTwoB}.  In
particular, by combining the upper bound~\eqref{EqnMeanBoundTermTwoB}
on the mean with the deviation
bound~\eqref{EqnDeviationBoundTermTwoB}, we find that
\begin{align*}
\Termtwob \leq \unicon \sqrt{ \frac{ \VarBound
    \log^2(\numobs)}{\numobs} } \: \sum_{j=1}^{\infty} \min \{ r^2,
\mu_{j} \radius^2 \} \qquad \mbox{with probability at least $1 -
  \numobs^{-10}$,}
\end{align*}
as claimed in equation~\eqref{EqnBoundTermTwoB}.


\subsubsection{Proof of inequality~\eqref{EqnMina}}
Now we focus on the first term $\Term_1 = \sup_{h \in
  \FuncClassStar(r)} \Big| \frac{1}{\numobs} \sum_{i=1}^{\numobs}
\noise_{i} \TruncLike{x_i} h(x_i) \Big|$.  
Repeating the same strategy as in the proof of the bound~\eqref{EqnBoundTermTwoB}, 
we see that 
\begin{align}
	\label{eq:T1}
	\Term_1  \leq \frac{1}{\numobs} \Bigsqrt{ 10 \sum_{j=1}^{\infty} \min\{r^2,\mu_{j} \radius^2 \}\cdot
    \big ( \sum_{i=1}^{\numobs} \noise_i  \TruncLike{x_i} \phi_{j} (x_i) \big)^2
     }.
\end{align}
Fix $\{x_i\}_{i=1}^{\numobs}$. We see that $\big ( \sum_{i=1}^{\numobs} \noise_i  \TruncLike{x_i} \phi_{j} (x_i) \big)^2$ is a quadratic form of independent sub-Gaussian random variables. Apply the Hanson-Wright inequality (e.g., Theorem 6.2.1 in the book~\cite{vershynin2018high}) to obtain that with probability at least $1 - \numobs^{-10}$, 
\begin{align}
	\label{eq:HW}
\big ( \sum_{i=1}^{\numobs} \noise_i  \TruncLike{x_i} \phi_{j} (x_i) \big)^2 \leq c_3
 \sigma^2 \sum_{i=1}^{\numobs} \big [   \TruncLike{x_i} \phi_{j} (x_i) \big] ^2. 
\end{align} 
It remains to control the term $\sum_{i=1}^{\numobs} \big [   \TruncLike{x_i} \phi_{j} (x_i) \big] ^2$. To this end, we invoke Bernstein's inequality to arrive at 
\begin{align*}
	\sum_{i=1}^{\numobs} \big [   \TruncLike{x_i} \phi_{j} (x_i) \big] ^2 \leq \Exp  \left[ \sum_{i=1}^{\numobs} \big [   \TruncLike{x_i} \phi_{j} (x_i) \big] ^2 \right] + c_4 \sqrt{ \alpha \log \numobs } +  c_5 \beta \log \numobs 
\end{align*}
with probability exceeding $1 - \numobs^{-10}$. Here, 
\begin{align*}
	\alpha &\coloneqq \Exp \sum_{i=1}^{\numobs} \mathsf{Var} \Big( \big [   \TruncLike{x_i} \phi_{j} (x_i) \big] ^2 \Big) \leq (\numobs \VarBound )^2 \\
	\beta &\coloneqq \sup_{x} | \big [   \TruncLike{x} \phi_{j} (x) \big] ^2| \leq \TruncLevel^2 = \numobs \VarBound, 
\end{align*}
are the variance and range statistics, respectively. This together with the upper bound \linebreak $
\Exp  \left[ \sum_{i=1}^{\numobs} \big [   \TruncLike{x_i} \phi_{j} (x_i) \big] ^2 \right] \leq \numobs \VarBound$ implies
\begin{align}
	\label{eq:Bernstein}
	\sum_{i=1}^{\numobs} \big [   \TruncLike{x_i} \phi_{j} (x_i) \big] ^2 \leq c_6 \numobs \VarBound \log \numobs. 
\end{align}
Combine the inequalities~\eqref{eq:T1}, \eqref{eq:HW}, and~\eqref{eq:Bernstein} to complete the proof of the inequality~\eqref{EqnMina}.

\section{Discussion}
\label{SecDiscussion}
In this paper, we study RKHS-based nonparametric regression under covariate shift. 
In particular, we focus on two broad families of covariate shift problems: 
(1) the uniformly $\LRBound$-bounded family, and (2) the $\chi^2$-bounded 
family. For the uniformly $\LRBound$-bounded family, we prove that the 
unweighted KRR estimate---with properly chosen regularization 
parameter---achieves optimal rate convergence for a large family of RKHSs with 
regular eigenvalues. In contrast, the na\"\i ve constrained kernel regression estimator 
is provably suboptimal under covariate shift. In addition, for the $\chi^2$-bounded 
family, we propose a likelihood-ratio-reweighted KRR with proper truncation 
that attains the minimax lower bound over this larger family of covariate shift problems. 

Our study is an initial step towards understanding the statistical nature of 
covariate shift. Below we single out several interesting directions to pursue in the 
future. First, it is of great importance to extend the study to other classes of regression 
functions, e.g., high dimensional linear regression, decision trees, etc. Second, while it 
is natural to measure discrepancy between source-target pairs using likelihood ratio, this 
is certainly not the only possibility. Various measures of discrepancy have been proposed in 
the literature, and it is interesting to see what the corresponding optimal procedures are. 
Thirdly, our upper and lower bounds match for regular
          kernels.  It is standard in the kernel regression literature
          to make an assumption regarding the decay of the kernel
          eigenvalue sequence \cite{caponnetto2007optimal,
            yang2017randomized}.  As highlighted by the corollaries to
          our main upper bound, the assumption of a regular kernel is
          general enough to capture the main examples of kernels used
          in practice.  Additionally, we emphasize that in this paper,
          we have adopted a worst-case perspective where we study the
          minimax rate of estimation for a sequence of regular kernel
          eigenvalues, over all $B$-bounded covariate shifts. A more
          instance-dependent perspective which studies these minimax
          rates for a fixed $B$-bounded covariate shift pair is very
          interesting and left for future work.   
Lastly, on a technical end, it is also interesting to see whether one can remove the uniform boundedness of the eigenfunctions in the unbounded likelihood ratio case, and retain the optimal 
rate of convergence.   In the current proof, we mainly use it to develop a localization bound~\eqref{eq:Linf-bound}  which guarantees that any function $h \in \mathcal{H}$ that is $r$-close to $f^\star$ in $\ell_2$ sense (roughly) enjoys an $\ell_\infty$ bound that also scales with $r$.


\subsection*{Acknowledgements}
This work was partially supported by NSF-DMS grant 2015454, NSF-IIS grant 1909365, as well as Office of Naval Research
grant DOD-ONR-N00014-18-1-2640 to MJW.
RP was partially supported by a Berkeley Fellowship and an ARCS Fellowship. 


\bibliographystyle{plain}

\bibliography{../../CS}

\appendix

\section{Proof of Theorem~\ref{ThmLowerBound}}
\label{sec:proof-lower-bound}

Let $\rlower$ be the smallest positive solution to the inequality 
$ \unicon' \sigma^2 \LRBound \frac{d(\delta)}{\numobs} \leq \delta^2$, 
where $\unicon' > 0$ is some large constant. 
We decompose the proof into two steps. First, we construct the lower
bound instance, namely the source, target distributions, and the
corresponding orthonormal basis.  Second, we apply the Fano method to
prove the lower bound.

\paragraph{Step 1: Constructing the lower bound instance.}

Let $\Target$ be a uniform distribution on $\{ \pm 1\}^{+\infty}$. For
the source distribution $\Source$, we set it as follows: with
probability $1 / \LRBound$, we sample $x$ uniformly on $\{ \pm
1\}^{+\infty}$, and with probability $1 - 1 / \LRBound$, we set $x =
0$. It can be verified that the pair $(\Source, \Target)$ has
$\LRBound$-bounded likelihood ratio. Corresponding to the target
distribution $\Target$, we take $\phi_j (x) = x_j$ for every $j \geq
1$.  In other words, we consider a linear kernel.

\paragraph{Step 2: Establishing the lower bound.}

In order to apply Fano's method, we first need to construct a packing
set of the function class $\Ball_\Hilbert(1)$.  For a given radius $r > 0$,
consider the $r$-localized ellipse
\begin{align*}
\Eset(r) \coloneqq \Big \{ \thetavec \mid \SUMJ \frac{ \theta_{j}^2 }{
  \min \{ r^2, \mu_j \} } \leq 1 \Big \}.
\end{align*}
It is straightforward to check that for any $\thetavec \in \Eset(r)$,
the function $f = \SUMJ \theta_{j} \phi_{j}$ lies in $\Ball_\Hilbert(1)$.
This set $\Eset(r)$ admits a large packing set in the $\ell_{2}$-norm,
as claimed in the following lemma.

\begin{lemma}
For any $r \in (0, \rlower]$, there exists a set $\{ \thetavec^1,
  \thetavec^2, \ldots, \thetavec^M \} \subseteq \Eset(r)$ with $\log M
  = \dlower / 64$ such that
\begin{align*}
  \| \thetavec^j - \thetavec^k \|_{2}^{2} \geq \tfrac{r^2}{4} \qquad
  \mbox{for any distinct pair of indices $j \neq k$.}
\end{align*}
\end{lemma}

\noindent See Lemma~4 in the paper~\cite{yang2017randomized}.

\medskip

Having constructed the packing, we then need to control the pairwise
KL divergence. Fix an index $j \in [M]$. Let $\Source \times
\ObsLike_{j}$ denote the joint distribution over the observed data $\{
(x_i, y_i) \}_{1 \leq i \leq \numobs}$ when the true function arises from 
$\thetavec^j$. Then for any pair of distinct indices $j
\neq k$, we have the upper bound
\begin{align*}
\mathsf{KL} ( \Source \times \ObsLike_{j} \| \Source \times
\ObsLike_{k}) = \frac{\numobs}{ 2 \sigma^2 } \cdot \Exp_{X \sim
  \Source} \Big[ \big( (\thetavec^j - \thetavec^k)^\top \phi(X)
  \big)^2 \Big] & \stackrel{(i)}{=} \frac{\numobs}{ 2 \sigma^2
  \LRBound } \| \thetavec^j - \thetavec^k \|_{2}^{2} \\
& \stackrel{(ii)}{\leq} \frac{ 2 \numobs r^2 }{\sigma^2 \LRBound},
\end{align*}
where step (i) follows from the definition of $\Source$; and step (ii)
follows from applying the triangle inequality, and the fact that
$\|\thetavec\|_2 \leq r$ for all $\thetavec \in \Eset(r)$.

Consequently, we arrive at the lower bound $\inf_{ \fhat } \sup_{
  \TrueFun \in \FuncClass } \Exp [ \| \fhat - \TrueFun
  \|_{\Target}^{2} ] \geq \frac{r^2}{8}$, valid for any sample size
satisfying the condition
\begin{align}
\label{eq:condition-lower-bound}
 \frac{ 2 \numobs r^2 }{ \sigma^2 \LRBound } + \log 2 \leq \frac{1}{2}
 \log M = \frac{\dlower}{128}.
\end{align}
By the definition of a regular kernel, we have $\dlower \geq c
\frac{\numobs \rlower^2 }{ \LRBound \sigma^2}$ for a universal
constant $\unicon$.  Furthermore, since $\rlower$ satisfies the
lower bound $\rlower^2 \geq \unicon' \frac{\sigma^2 \LRBound}{ \numobs}$,
the condition~\eqref{eq:condition-lower-bound} is met by setting $r^2
= \unicon_1 \rlower^2$ for some sufficiently small constant $\unicon_1 > 0$. 



\section{Proof of Theorem~\ref{ThmConstrainedKRR}}
\label{sec:proof-failure}

Let the sample size $\numobs \geq 1$ and likelihood ratio bound
$\LRBound \geq 1$ be given.  Our failure instance relies on a function class
$\FuncClass_\numobs$, together with a pair of distributions $(\Source,
\Target)$.  The function class $\FuncClass_\numobs$ is the unit ball
of a RKHS with finite-rank kernel, over the hypercube $\{-1,
+1\}^\numobs$.  The kernel is given by $ \KerFun(x, z) \coloneqq
\sum_{j=1}^\numobs \mu_j \phi_j(x) \phi_j(z)$. The eigenfunctions and
eigenvalues are
\begin{align*}
\phi_j(x) = x_j, \quad \mbox{and} \quad \mu_j = \frac{1}{j^2}, \quad \mbox{for}~j=1, \ldots, \numobs.
\end{align*}
To be clear, the function class is given by
\begin{align*}
\FuncClass_\numobs \coloneqq \{\, f \coloneqq \sum_{j=1}^\numobs
\theta_j \phi_j \mid \sum_{j=1}^\numobs \tfrac{\theta_j^2}{\mu_j} \leq
1 \,\}.
\end{align*}
The target distribution, $\Target$, is the uniform distribution on
$\{-1, +1\}^\numobs$.  The source distribution is a product
distribution, $\Source = \otimes_{j=1}^\numobs \Source_j$.  We take
$\Source_j$ to be uniform on $\{+1, -1\}$, when $j > 1$.  On the other
hand, the first coordinate follows the distribution
\begin{align*}
\Source_1 \coloneqq \Big(1 - \frac{1}{\LRBound}\Big)\delta_0 +
\frac{1}{\LRBound} \mathrm{Unif}(\{-1, +1\}).
\end{align*}
It is immediate that $(\Source, \Target)$ have $\LRBound$-bounded
likelihood ratio.

Given this set-up, our first step is to reduce the lower bound to the
separation of a single coordinate of the parameter associated with the
empirical risk minimizer and a single coordinate of the parameter
associated with a hard instance in the function class of interest
$\FuncClass_\numobs$.  We introduce a one-dimensional minimization
problem that governs this separation problem and allows us to
establish our result.

\subsection{Reduction to a one dimensional separation problem}

To establish our lower bound it suffices to consider the following
``hard'' function
\[
\hardfn(x) = x_1 = \sum_{j=1}^\numobs (\hardtheta)_j \phi_j(x),
\quad \mbox{where} \quad
\hardtheta = (1, 0, \dots, 0) \in \Real^\numobs.
\]
Since $\phi_j(x) = x_j$ and $\mu_j = j^{-2}$, it follows that
$\hardfn\in \FuncClass_\numobs$. We can write
$\fhaterm(x) = \sum_{j=1}^\numobs (\thetahaterm)_j x_j$, where we defined
\begin{equation}\label{disp:theta-erm}
\thetahaterm \coloneqq \arg \min\Big\{\,
\sum_{i=1}^\numobs \big(\sum_{j=1}^\numobs \theta_j x_{ij} - y_i\big)^2
\mid \sum_{j=1}^\numobs \tfrac{\theta_j^2}{\mu_j} \leq 1\,\Big\}.
\end{equation}
Putting these pieces together, we see that
\begin{equation}\label{ineq:lower-reduce-to-first-coord}
  \sup_{\TrueFun \in \FuncClass_\numobs}
  \Exp \Big[\| \fhaterm - \TrueFun \|_{\Target}^{2}\Big] 
  \geq
  \Exp \Big[\| \fhaterm - \hardfn \|_{\Target}^{2}\Big] 
  \stackrel{{\rm(i)}}{=}
  \Exp \Big[ \|\thetahaterm - \hardtheta\|_2^2\Big]
  \stackrel{{\rm(ii)}}{\geq}
  \Exp \Big[\big((\thetahaterm)_1 - \thetastar_1\big)^2\Big].
\end{equation}
Above, the relation (i) is a consequence of Parseval's theorem,
along with the orthonormality of $\{\phi_j\}_{j=1}^\numobs$ in
$L^2(\Target)$. Inequality (ii) follows by dropping terms corresponding to indices 
indices $j > 1$. Therefore,
in view of display~\eqref{ineq:lower-reduce-to-first-coord},
it suffices to show that:
\begin{equation}\label{ineq:lower-const-prob}
\Prob\Big\{\big((\thetahaterm)_1 - 1\big)^2\geq c_3 \frac{B^3}{\numobs^2}\Big\} \geq \frac{1}{2}.
\end{equation}

\subsection{Proof of one-dimensional separation bound~\eqref{ineq:lower-const-prob}}

We begin with a proof outline.
  
\paragraph{Proof outline}

To establish~\eqref{ineq:lower-const-prob}, we can assume 
$(\thetahaterm)_1 \in [0, 1]$; otherwise, the lower bound follows trivially, provided
$c_2$ is sufficiently small, in particular, $c_2 \leq \sqrt[3]{1/c_3}$.
We introduce a bit of notation:
\[
\EmpCov_{\Source} \coloneqq \frac{1}{\numobs} \sum_{i=1}^\numobs x_i x_i^\top
\quad \mbox{and} \quad
\corvec \coloneqq \frac{1}{\numobs} \sum_{i=1}^\numobs w_i x_i.
\]
Thus, we can further restrict the empirical risk minimization problem~\eqref{disp:theta-erm} to
\begin{align}
\thetatild &\coloneqq \arg \min\Big\{\,
\sum_{i=1}^\numobs \big(\sum_{j=1}^\numobs \theta_j x_{ij} - y_i\big)^2
\mid \sum_{j=1}^\numobs \frac{\theta_j^2}{\mu_j} \leq 1,~~\theta_1 \in [0, 1]\,\Big\} \nonumber\\
&= \arg \min\bigg\{\,
(\theta - \thetastar)^T \EmpCov_{\Source} (\theta - \thetastar) - 2 \corvec^T (\theta - \thetastar)
\mid \sum_{j=1}^\numobs \frac{\theta_j^2}{\mu_j} \leq 1,~~\theta_1 \in [0, 1]\,\bigg\} \label{disp:theta-tild}.
\end{align}
Indeed, in order to prove inequality~\eqref{ineq:lower-const-prob}, it
suffices to show that
\begin{equation}
\label{ineq:lower-const-prob-suff}
\Prob\Big\{\big(\thetatild_1 - 1\big)^2\geq c_3
\frac{B^3}{\numobs^2}\Big\} \geq \frac{1}{2}.
\end{equation}
Let us define an auxiliary function $\keyfunc \colon [0, 1] \to
\Real$, given by
\begin{equation}
\label{defn:key-fun}
\keyfunc(t) \coloneqq \inf\Big\{ (\theta - \thetastar)^T
\EmpCov_{\Source} (\theta - \thetastar) - 2 \corvec^T (\theta -
\thetastar) \mid \sum_{j=1}^\numobs \frac{\theta_j^2}{\mu_j} \leq
1,~~\theta_1 = t \Big\}.
\end{equation}
By definition~\eqref{disp:theta-tild}, the choice $\thetatild$
minimizes this objective, and therefore $\inf_{t \in [0, 1]} g(t) =
g(\tilde \theta_1)$.  The next two lemmas concern the minimum value
and minimizer of $\keyfunc$.  Lemma~\ref{lemma:inf-value}, which we
prove in section~\ref{sec:proof-inf-value}, bounds the minimal value
from above.  Lemma~\ref{lemma:impossible-set}, demonstrates that there
is an interval of length order $\sqrt{B^3/\numobs^2}$ on which the
function $\keyfunc$ is bounded away from the minimal value. We prove
this result in Section~\ref{sec:proof-impossible-set}.

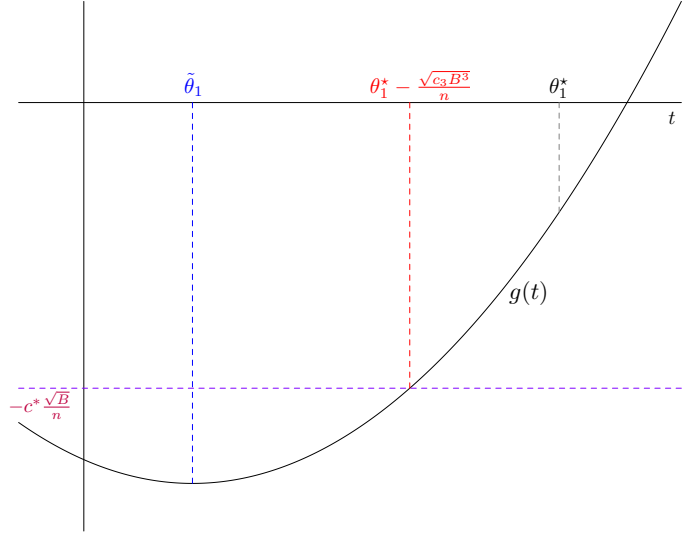
\begin{figure}[h!t]
\centering
\begin{overpic}[scale=0.5,unit=1mm]{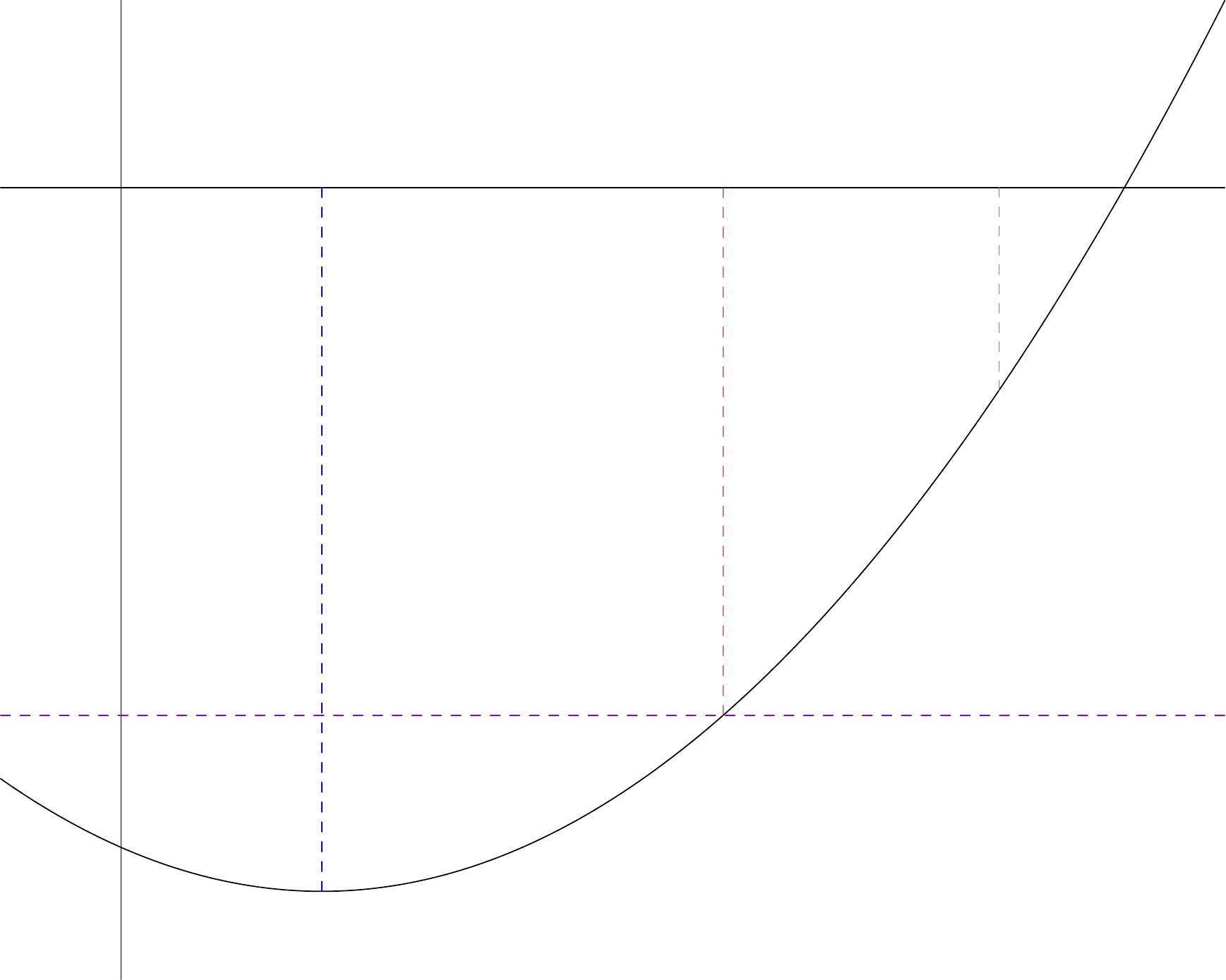}
  \put(74, 35){\scalebox{0.9}{$g(t)$}} \put(-1.5,
  18){\scalebox{0.75}{\color{purple}
      $-\cstar\frac{\sqrt{B}}{\numobs}$}} \put(98,
  61.5){\scalebox{0.75}{$t$}} \put(25,
  66.2){\scalebox{0.8}{\color{blue} $\thetatild_1$}} \put(53,
  66.2){\scalebox{0.8}{\color{red} $\thetastar_1 - \tfrac{\sqrt{c_3
          B^3}}{n}$}} \put(80, 66.2){\scalebox{0.8}{{$\thetastar_1$}}}
\end{overpic}
\caption{Pictorial representation of lower bound argument, separating
  the first coordinate of empirical risk minimizer, $\thetatild_1$,
  from the true population minimizer $\theta^\star_1$.
  Lemma~\ref{lemma:inf-value} establishes the upper bound, depicted in
  purple above, on the minimal value of $\keyfunc$.
  Lemma~\ref{lemma:impossible-set} establishes an interval, shown
  between the red dashed line and $\theta^\star_1$ above, which
  excludes $\thetatild_1$. This allows us to ensure that
  $\theta^\star_1$ and $\thetatild_1$ are sufficiently separated.}
\label{fig:Separation}
\end{figure}

\begin{lemma} [Minimal value of empirical objective]
\label{lemma:inf-value}
There is a constant $\cstar > 0$ such that
\begin{align*}
\keyfunc(\thetatild_1) & \leq - \cstar
\frac{\sqrt{\LRBound}}{\numobs}
\end{align*}
holds with probability at least $3/4$.
\end{lemma}

\begin{lemma} [Separation from $\theta^\star_1$]
\label{lemma:impossible-set}
There exists a constant $c_3 > 0$ such that
\begin{align}
\label{EqnKeyLower}
\inf_{\substack{t \in [0, 1] \\
    (1 -t)^2 \leq c_3 B^3/n^2}}
\keyfunc(t)
>
-\cstar \frac{\sqrt{\LRBound}}{\numobs}.
\end{align}
where probability at least $3/4$.
\end{lemma}

Note that the constant $\cstar$ used in Lemmas~\ref{lemma:inf-value}
and~\ref{lemma:impossible-set} is the same.
Thus---after union bounding over the two error events---with
probability at least $1/2$,
\[
\keyfunc(\thetatild_1) < \inf_{\substack{t \in [0, 1] \\
    (1 -t)^2 \leq c_3 B^3/n^2}}
\keyfunc(t).
\]
Recalling that $\thetatild_1 \in [0, 1]$, we
conclude on this event that $(1 - \thetatild_1)^2 \geq c_3 \tfrac{B^3}{\numobs^2}$,
which furnishes~\eqref{ineq:lower-const-prob-suff}, and
thereby establishes the required result.
To complete the proof, it then remains to establish the auxiliary lemmas
stated above. Before doing so, we record a useful lemma, which will be used 
multiple times later.

\subsection{A useful lemma}
\begin{lemma}\label{lemma:effective-dimension}
For any quantity $\Quantity \in (
\frac{\LRBound}{ 4 \numobs^2 }, \frac{\LRBound}{ 4 }) $, with
probability at least $1 - c_{1} \exp \big( - c_{2}
\frac{\LRBound^{1/2}}{\Quantity^{1/2}} \big)$, one has
\begin{align*}
c \frac{\LRBound^{1/2}}{\Quantity^{1/2}} \frac{1}{\numobs} \leq 
\sum_{j=2}^{\numobs} \frac{ ( \corvec_{j} )^{2} }{ 1  + \frac{\Quantity}{ \LRBound \mu_{j}} } 
\leq C \frac{\LRBound^{1/2}}{\Quantity^{1/2}} \frac{1}{\numobs} .
\end{align*}
Here $c_{1}, c_{2}, C, c > 0$ are absolute constants.
\end{lemma}

\begin{proof}
For each $j \geq 1$, define $ \eta_{j} \coloneqq \big( 1 +
\frac{\Quantity}{ \LRBound \mu_{j} } \big)^{-1}$.  We focus on
controlling the term
\begin{align*} 
\sum_{j=2}^{\numobs} \eta_{j} \left[( \sqrt{\numobs}  \corvec_{j}  )^{2} - 1\right].
\end{align*}
Recall from the definition of $v$ that $ \corvec_{j} =
\frac{1}{\numobs} \sum_{i=1}^{\numobs} \noise_{i} x_{ij}.  $ Under the
construction of the lower bound instance, we have $\sqrt{\numobs}
\corvec_{j} \stackrel{\mathsf{i.i.d.}}{\sim} \mathcal{N}(0,1)$.
Therefore $\sqrt{\numobs} \corvec_{j} )^{2} - 1$ is a mean-zero
sub-exponential random variable. This allows us to invoke Bernstein's
inequality to obtain
\begin{align*}
\Prob \left( \left| \sum_{j=2}^{\numobs} \eta_{j} \left[ (
  \sqrt{\numobs} \corvec_{j} )^{2} - 1 \right] \right| \geq t \right)
\leq 2 \exp \left\{ - c \min \left( \frac{ t^{2} }{ \sum_{j \geq 2}
  \eta_{j}^{2} }, \frac{t}{ \max_{j} \eta_{j} } \right) \right\},
\end{align*}
where $c > 0$ is some universal constant. 

We claim that there exist three constants $C_1, C_2, C_3 > 0$ such
that
\begin{subequations}\label{subeq:eta}
\begin{align}
\max_{j = 2,\ldots, \numobs} \eta_{j} & \leq 1; \label{eq:max-eta}
\\ \sum_{j=2}^{\numobs} \eta_{j}^{2} & \leq C_1
\frac{\LRBound^{1/2}}{\Quantity^{1/2}}; \label{eq:variance-eta} \\ C_2
\frac{\LRBound^{1/2}}{\Quantity^{1/2}} \leq \sum_{j=2}^{\numobs}
\eta_{j} & \leq C_3
\frac{\LRBound^{1/2}}{\Quantity^{1/2}} \label{eq:sum-eta}.
\end{align}
\end{subequations}
As a result, we can $t = c_0 \frac{\LRBound^{1/2}}{\Quantity^{1/2}}$
with $c_0$ sufficiently small to arrive at the desired conclusion.

We are left with proving the claimed relations~\eqref{subeq:eta}.  The
first relation~\eqref{eq:max-eta} is trivial. We provide the proof of
the third inequalities~\eqref{eq:sum-eta}; the proof of the middle one
(cf.~relation~\eqref{eq:variance-eta}) follows by a similar
argument. Since $\Quantity \in ( \frac{\LRBound}{ 4 \numobs^2 },
\frac{\LRBound}{ 4 }) $, we can decompose the sum into
\begin{align*}
        \sum_{j=2}^{\numobs} \eta_{j} = \sum_{j=2}^{\lfloor \sqrt{
            \LRBound / \Quantity} \rfloor } \frac{ 1 }{ 1 +
          \frac{\Quantity}{ \LRBound \mu_{j}} } + \sum_{j= \lfloor
          \sqrt{ \LRBound / \Quantity} \rfloor + 1}^{n} \frac{ 1 }{ 1
          + \frac{\Quantity}{ \LRBound \mu_{j}} }.
\end{align*}
Recall that $\mu_{j} = j^{-2}$. We thus have $1 \geq \frac{\Quantity}{
  \LRBound \mu_{j}}$ for $j \leq \lfloor \sqrt{ \LRBound / \Quantity}
\rfloor$ and $1 \leq \frac{\Quantity}{ \LRBound \mu_{j}}$ for $j \geq
\lfloor \sqrt{ \LRBound / \Quantity} \rfloor$. These allow us to upper
bound $\sum_{j=2}^{\numobs} \eta_{j}$ as
\begin{align*}
\sum_{j=2}^{\numobs} \eta_{j} \leq \lfloor \sqrt{ \LRBound /
  \Quantity} \rfloor + \frac{\LRBound}{\Quantity} \sum_{j= \lfloor
  \sqrt{ \LRBound / \Quantity} \rfloor + 1}^{n} \frac{ 1 }{ j^2 } \leq
C_3 \frac{\LRBound^{1/2}}{\Quantity^{1/2}}.
\end{align*}
Similarly, we have the lower bound 
\begin{align*}
\sum_{j=2}^{\numobs} \eta_{j} \geq \sum_{j=2}^{\lfloor \sqrt{ \LRBound
    / \Quantity} \rfloor } \frac{ 1 }{ 1 + \frac{\Quantity}{ \LRBound
    \mu_{j}} } \geq \frac{1}{2} \lfloor \sqrt{ \LRBound / \Quantity}
\rfloor \geq C_2 \frac{\LRBound^{1/2}}{\Quantity^{1/2}}.
\end{align*}
This finishes the proof. 
\end{proof}


\subsection{Proof of auxiliary lemmas}

In order to facilitate the proofs of these lemmas, it is useful to
decompose $\theta = (\theta_1, \thetavecRem) \in \Real \times
\Real^{\numobs - 1}$.  Additionally, we consider the constraint set
\begin{align*}
\Constraint(\firstcoord) \coloneqq\Big\{ \thetavecRem \in
\Real^{\numobs - 1} \mid \sum_{j=2}^\numobs \frac{\theta_j^2}{\mu_j}
\leq 1 - \firstcoord^2\Big\}, \quad \mbox{where}~\firstcoord \in [0,
  1].
\end{align*}
This set plays a key role. In view of definition~\eqref{defn:key-fun},
we can write
\begin{equation}
  \label{eqn:key-fun-constraint-set}
  \keyfunc(\firstcoord) =
  \inf_{\thetavecRem \in \Constraint(\firstcoord)}\,
  \left
  \{
  \begin{bmatrix}
    \firstcoord-1 \\ \thetavecRem
  \end{bmatrix}^{\top}
  \EmpCov_{\Source}
  \begin{bmatrix}
    \firstcoord-1\\ \thetavecRem
  \end{bmatrix}
  - 2
  \begin{bmatrix}
    \firstcoord-1\\
    \thetavecRem
  \end{bmatrix}^\top
  \corvec
    \right
  \},
\end{equation}
where above we have used $\theta^\star = (1, 0, \dots,0)$.
Finally, we will use the diagonal matrix of
kernel eigenvalues $\Mmat \coloneqq \diag(\mu_1, \mu_2, \dots, \mu_\numobs)$,
repeatedly. 

\subsubsection{Proof of Lemma~\ref{lemma:inf-value}}
\label{sec:proof-inf-value}
We show that with probability at least $3/4$,
\begin{equation}\label{omega-condition}
\keyfunc(\uniVar) \leq - \cstar \frac{\sqrt{\LRBound}}{\numobs},
\quad \mbox{where} \quad
\uniVar \coloneqq \sqrt{1 - \frac{\LRBound^{3/2}}{\numobs}}. 
\end{equation}
When $\numobs^2 \geq \LRBound^{3}$, we have $\uniVar \in [0, 1]$.
Since $\inf_{t \in [0, 1]} \keyfunc(t) \leq \keyfunc(\uniVar)$,
the display~\eqref{omega-condition} implies the result.

\paragraph{Proof of bound~\eqref{omega-condition}:}
From the proof of
Lemma~\ref{LemBasicCov}, if we set $\lambda \coloneqq C \frac{\log
  \numobs}{\numobs} $ for some constant $C > 0$, then we have
\begin{align}
\label{eq:key-event-erm}
\tfrac{1}{2}( \PopCov_{\Source} + \lambda \Mmat^{-1} ) \preceq
\EmpCov_{\Source} + \lambda \Mmat^{-1} \preceq \tfrac{3}{2} (
\PopCov_{\Source} + \lambda \Mmat^{-1} ),
\end{align}
with probability at least $1 - \tfrac{1}{\numobs}$.
Consequently, for any vector $\thetavec$ obeying $\thetavec^{\top}
\Mmat^{-1} \thetavec \leq 1$, we have the upper bound
\begin{align*}
\left( \thetavec - \thetavec^{\star} \right)^{\top} \EmpCov_{\Source}
\left( \thetavec - \thetavec^{\star}\right) & = \left( \thetavec -
\thetavec^{\star} \right)^{\top} \left( \EmpCov_{\Source}+\lambda
\Mmat^{-1} \right) \left( \thetavec - \thetavec^{\star} \right) -
\lambda \left( \thetavec - \thetavec^{\star} \right)^{\top} \Mmat^{-1}
\left( \thetavec - \thetavec^{\star} \right) \\ & \leq \frac{3}{2}
\left( \thetavec - \thetavec^{\star} \right)^{\top} (
\PopCov_{\Source} + \lambda \Mmat^{-1} ) \left( \thetavec -
\thetavec^{\star} \right) - \lambda \left( \thetavec -
\thetavec^{\star} \right)^{\top} \Mmat^{-1} \left( \thetavec -
\thetavec^{\star} \right) \\ & = \frac{3}{2} \left( \thetavec -
\thetavec^{\star} \right)^{\top} \PopCov_{\Source} \left( \thetavec -
\thetavec^{\star} \right) + \frac{\lambda}{2} \left( \thetavec -
\thetavec^{\star} \right)^{\top} \Mmat^{-1} \left( \thetavec -
\thetavec^{\star} \right)\\ & \leq \frac{3}{2} \left( \thetavec -
\thetavec^{\star} \right)^{\top} \PopCov_{\Source} \left( \thetavec -
\thetavec^{\star} \right) + 2 \lambda,
\end{align*}
where the final inequality holds since $\left( \thetavec -
\thetavec^{\star} \right)^{\top} \Mmat^{-1} \left( \thetavec -
\thetavec^{\star} \right) \leq 4$.
Applying this result with the
vector $\thetavec = ( \uniVar , \thetavecRem)^\top$ yields
\begin{align}
\keyfunc(\uniVar) & \leq \min_{\thetavecRem \in \Constraint} \left \{
\frac{3}{2} \left[\begin{array}{c} \uniVar-1\\ \thetavecRem
\end{array}\right]^{\top} \PopCov_{\Source} \left[\begin{array}{c}
\uniVar-1\\ \thetavecRem
\end{array}\right]-2\left[\begin{array}{c}
\uniVar-1\\ \thetavecRem
  \end{array}\right]^{\top} \corvec + 2 \lambda \right \}  \nonumber\\
&= \Term_1(\uniVar) + \Term_2(\uniVar) + 2 \lambda +
\min_{\thetavecRem \in \Constraint} \Term_3(\thetavecRem).
\label{ineq:g-omega-upper}
\end{align}
Above, we have defined
\begin{align}
\Term_1(\uniVar) \coloneqq \frac{3}{2}
\frac{(\uniVar-1)^{2}}{\LRBound}, \quad \Term_2(\uniVar) \coloneqq -
2 \corvec_1(\uniVar-1), \quad \mbox{and} \quad \Term_3(\thetavecRem)
\coloneqq \frac{3}{2}\|\thetavecRem\|_{2}^{2} - 2 \corvecRem^{\top}
\thetavecRem,\label{eq:def-T}
\end{align}
and we have used the decomposition $\corvec = (
  \corvec_1, \corvecRem)^\top$.
We now bound each of these three terms in turn.

\paragraph{Controlling the term $\Term_1(\uniVar)$:}

Recall $\uniVar \in [0, 1]$ satisfies the equality $1 - \uniVar^{2} =
\frac{\LRBound^{3/2}}{\numobs}$. Consequently, we have
\begin{align} \label{ineq:t1-upper-final}
  \Term_1(\uniVar) = \frac{3}{2} \frac{(1 - \uniVar)^{2}}{\LRBound} \leq
  \frac{3}{2} \frac{(1 - \uniVar^2)^{2}}{\LRBound}
= \frac{3}{2} \frac{\LRBound^{2}}{\numobs^{2}}.
\end{align}

\paragraph{Controlling the term $\Term_2(\uniVar)$: }
For the second term, by definition of $\uniVar$, we have
\begin{align}
  \label{ineq:t2-upper}
  \Term_2(\uniVar )= 2 \corvec_1 ( 1 -\uniVar)
  \leq 2 |\corvec_1 | (1 - \uniVar) \leq 2 | \corvec_1 | (1 - \uniVar^2)
  =  2 |\corvec_1 | \cdot
  \frac{\LRBound^{3/2}}{\numobs}.
\end{align}
We have the following lemma to control the size of $|\corvec_1|$.

\begin{lemma}\label{lemma:second-moment}
The
following holds true with probability at least 0.99
\begin{align*}
\left| \frac{1}{\numobs} \sum_{i=1}^{\numobs} \noise_{i} x_{i1} \right| 
\leq \frac{10}{ \sqrt{\numobs \LRBound} }.
\end{align*}
\end{lemma}

\begin{proof}
In view of the construction of the lower bound instance, we can
calculate
\begin{align*}
\Exp \left[ \left( \frac{1}{\numobs} \sum_{i=1}^{\numobs} \noise_{i} x_{i1} 
\right)^{2} \right] = \frac{1}{\numobs^{2}} \sum_{i=1}^{\numobs} \Exp 
\left[ \noise_{i}^{2} x_{i1}^{2} \right] = \frac{1}{ \numobs \LRBound }.
\end{align*}
The claim then follows from Chebyshev's inequality. \end{proof}

Lemma~\ref{lemma:second-moment} demonstrates that
$| \corvec_1 | \leq \tfrac{10}{\sqrt{\numobs
    \LRBound}}$, with probability at least $99/100$. 
Therefore, on this event, the bound~\eqref{ineq:t2-upper} guarantees
\begin{equation}\label{ineq:t2-upper-final}
  \Term_2(\uniVar) \leq 20 \frac{B}{\numobs^{3/2}}.
\end{equation}


\paragraph{Controlling the term $\Term_3(\theta_r)$: }

Our final step is to upper bound the constrained minimum $\min
\limits_{\thetavecRem \in \Constraint} \Term_3(\thetavecRem)$.  Since
this minimization problem is strictly feasible, Lagrange duality
guarantees that
\begin{align*}
\min_{\thetavecRem \in \Constraint} \Term_3(\thetavecRem) & = \min_{
  \thetavecRem} \max_{\dualVar \geq 0} \left\{ \frac{3}{2} \|
\thetavecRem \|_{2}^{2} - 2 \corvecRem^{\top} \thetavecRem + \dualVar (
\thetavecRem^{\top} \MmatRem^{-1} \thetavecRem -
\frac{\LRBound^{3/2}}{\numobs} ) \right\} \\
& = \max_{\dualVar \geq 0} \min_{\thetavecRem} \left\{ \frac{3}{2}
\|\thetavecRem\|_{2}^{2} - 2 \corvecRem^{\top} \thetavecRem + \dualVar
( \thetavecRem^{\top} \MmatRem^{-1}\thetavecRem -
\frac{\LRBound^{3/2}}{\numobs})\right\} .
\end{align*}
The inner minimum is achieved at $\thetavecRem = \left[ \frac{3}{2}
  \IdMat + \dualVar \MmatRem^{-1} \right]^{-1} \corvecRem$, so that
we have established the equality
\begin{align*}
\min_{\thetavecRem \in \Constraint} \Term_3(\thetavecRem) & = \max_{
  \dualVar \geq 0 } \left\{ - \dualVar \tfrac{\LRBound^{3/2}}{\numobs}
- \corvecRem^{\top} \left[ \tfrac{3}{2} \IdMat + \dualVar
  \MmatRem^{-1} \right]^{-1} \corvecRem \right\} \; = \;
\max_{\dualVar \geq 0} \left\{ -\dualVar
\frac{\LRBound^{3/2}}{\numobs} - \sum_{j=2}^{\numobs} \frac{(
  \corvec_{j} )^{2}}{\frac{3}{2} + \tfrac{\dualVar}{\mu_{j}}}
\right\} .
\end{align*}

It remains to analyze the maximum over the dual variable $\dualVar$,
and we split the analysis into two cases.
\begin{itemize}
\item Case 1: First, suppose that the maximum is achieved at some
  $\dualVar^{\star} \geq \tfrac{1}{\LRBound}$.  In this case, we have
\begin{align*}
\max_{\dualVar \geq 0} \left\{ - \dualVar
\tfrac{\LRBound^{3/2}}{\numobs} - \sum_{j=2}^{\numobs} \frac{(
  \corvec_{j} )^{2}}{\tfrac{3}{2} + \tfrac{\dualVar}{\mu_{j}}} \right
\} \leq - \dualVar^{\star} \frac{\LRBound^{3/2}}{\numobs} \leq -
\frac{\LRBound^{1/2}}{\numobs}.
\end{align*}
\item Case 2: Otherwise, we may assume that the maximum achieved at
  some $\dualVar^{\star} \in [0, \tfrac{1}{\LRBound}]$, in which case
  we have
\begin{align*}
\max_{\dualVar \geq 0} \left\{ - \dualVar
\tfrac{\LRBound^{3/2}}{\numobs} - \sum_{j=2}^{\numobs} \frac{(
  \corvec_{j} )^{2}}{\tfrac{3}{2} + \tfrac{\dualVar}{\mu_{j}}}
\right\} \leq - \sum_{j=2}^{\numobs} \frac{( \corvec_{j}
  )^{2}}{\tfrac{3}{2} + \tfrac{\dualVar^{\star}}{\mu_{j}}} \leq
-\sum_{j=2}^{\numobs} \frac{( \corvec_{j} )^{2}}{\tfrac{3}{2} +
  \frac{1}{\LRBound \mu_{j}}} \leq -c \frac{\LRBound^{1/2}}{\numobs},
\end{align*}
where $c > 0$ is a constant. Here in view of
Lemma~\ref{lemma:effective-dimension}, the last inequality holds with
probability at least 0.9 as long as $\LRBound$ is sufficiently large.
\end{itemize}
Combining the two cases, we arrive at the conclusion that as long as 
$\LRBound$ is sufficiently large, with probability at least 0.9, 
\begin{equation}
  \label{ineq:t3-upper-final}
\min_{\thetavecRem \in \Constraint} \Term_3(\thetavecRem) 
\leq -c_1 \frac{\LRBound^{1/2}}{\numobs}
\end{equation}
for some constant $c_1 > 0$.

\paragraph{Completing the proof:}
We can now combine bounds~\eqref{ineq:t1-upper-final},~\eqref{ineq:t2-upper-final},
and~\eqref{ineq:t3-upper-final} on the terms $\Term_1, \Term_2, \Term_3$,
respectively.
Note that when $\numobs \geq 7 B^{3/2} \geq 7$,
all three events and the upper bound~\eqref{ineq:g-omega-upper} hold simultaneously,
with probability
$1 - (\tfrac{1}{\numobs} + \tfrac{1}{100} + \tfrac{1}{10}) \geq 3/4$.
Therefore, we obtain
\begin{align*}
\keyfunc(\uniVar) & \leq \frac{3}{2}
\frac{\LRBound^{2}}{\numobs^{2}} + 20 \frac{\LRBound}{\numobs^{3/2}}
- c_1 \frac{\LRBound^{1/2}}{\numobs} + C \frac{\log \numobs}{\numobs}
\\
 &\leq - \frac{c_1}{2} \frac{\LRBound^{1/2}}{\numobs}.
\end{align*}
The final inequality above holds, since 
$\LRBound \geq c_1 (\log \numobs)^2$ and $\numobs 
\geq 7\LRBound^{3/2}$, for sufficiently large $c_1 > 0$.

\subsubsection{Proof of Lemma~\ref{lemma:impossible-set}}
\label{sec:proof-impossible-set}

We will prove the slightly stronger claim that with probability at least
$3/4$, we have
\begin{equation}\label{ineq:stronger-lower}
\inf_{\substack{t \in [0, 1] \\
    1 -t^2 \leq \beta B^{3/2}/n}}
\keyfunc(t)
>
-\cstar \frac{\sqrt{\LRBound}}{\numobs}
\end{equation}
To see that this proves the claim, note that
$\sup_{t \in [0, 1]} \tfrac{(1-t^2)^2}{(1-t)^2} = 4$
Therefore, if $(1-t)^2 \leq \tfrac{\beta^2}{4} \tfrac{\LRBound^3}{\numobs^2}$,
then $(1-t^2)^2 \leq \beta^2 \tfrac{\LRBound^3}{\numobs^2}$.
Hence,~\eqref{ineq:stronger-lower} proves the claim as soon as
$c_3 = \beta^2/4$. 

\paragraph{Proof of bound~\eqref{ineq:stronger-lower}:}
On the event~\eqref{eq:key-event-erm},
if $\thetavec=(\thetavec_1, \thetavecRem)^{\top}$ obeys
$\thetavec^{\top} \Mmat^{-1} \thetavec \leq 1$, then we have the lower
bound
\begin{align*}
\left( \thetavec - \thetavec^{\star} \right)^{\top}\EmpCov_{\Source}
\left( \thetavec - \thetavec^{\star} \right) & =\left( \thetavec -
\thetavec^{\star} \right)^{\top}\left(\EmpCov_{\Source} + \lambda
\Mmat^{-1} \right) \left( \thetavec - \thetavec^{\star} \right) -
\lambda \left( \thetavec - \thetavec^{\star} \right)^{\top} \Mmat^{-1}
\left( \thetavec - \thetavec^{\star} \right) \\ & \geq
\frac{1}{2}\left( \thetavec - \thetavec^{\star} \right)^{\top} (
\PopCov_{\Source} +\lambda \Mmat^{-1} ) \left( \thetavec -
\thetavec^{\star} \right)-\lambda\left( \thetavec - \thetavec^{\star}
\right)^{\top} \Mmat^{-1} \left( \thetavec - \thetavec^{\star} \right)
\\
& = \frac{1}{2} \left( \thetavec - \thetavec^{\star} \right)^{\top}
\PopCov_{\Source} \left( \thetavec - \thetavec^{\star} \right) -
\frac{\lambda}{2}\left( \thetavec - \thetavec^{\star} \right)^{\top}
\Mmat^{-1} \left( \thetavec - \thetavec^{\star} \right)\\ & \geq
\frac{1}{2}\left( \thetavec - \thetavec^{\star} \right)^{\top}
\PopCov_{\Source} \left( \thetavec - \thetavec^{\star} \right) - 2
\lambda,
\end{align*}
valid when $\lambda = C \tfrac{\log \numobs}{\numobs} $ for some
constant $C > 0$.  Consequently, we have
\begin{align}
\keyfunc(\thetavec_1) & \geq \min_{\thetavecRem \in
  \Constraint(\thetavec_1)} \left \{ \frac{1}{2}\left( \thetavec
- \thetavec^{\star} \right)^{\top} \PopCov_{\Source} \left( \thetavec
- \thetavec^{\star} \right) - 2 \left( \thetavec - \thetavec^{\star}
\right)^{\top}\corvec - 2 \lambda \right \} \nonumber \\
& = \min_{\thetavecRem \in \Constraint(\thetavec_1)} \left \{
\frac{1}{2}\frac{(\thetavec_1 - 1)^{2}}{\LRBound} - 2
\corvec_1(\thetavec_1-1) + \frac{1}{2} \|\thetavecRem\|_{2}^{2} - 2
\corvecRem^{\top} \thetavecRem - 2 \lambda \right \} \nonumber \\
& \geq - \Term_2(\thetavec_1) - 2 \lambda 
+ \min_{\thetavecRem \in \Constraint(\thetavec_1)} 
\{ \frac{1}{2} \|\thetavecRem\|_{2}^{2} - 2
\corvecRem^{\top} \thetavecRem \}. \label{eq:lower-bound-g}
\end{align}
where the last line identifies $- 2
\corvec_1(\thetavec_1-1) $ with $\Term_2(\thetavec_1)$ 
(cf.~definition~\eqref{eq:def-T}).

We separate the proof into two cases---mainly to get around 
the duality issue. 

\paragraph{Case 1: $\thetavec_1=1$. } In this case, we have
\begin{align}\label{eq:boundary-point}
\keyfunc(\thetavec_1) \geq - 2 \lambda = - \frac{2C \log
  \numobs}{\numobs}.
\end{align}

\paragraph{Case 2: $\thetavec_1 \in [0,1)$. } 
We lower bound the terms in equation~\eqref{eq:lower-bound-g} 
in turn. 

\begin{itemize}
\item {\bf Lower bounding $\Term_2(\thetavec_1)$.} 
For any $0 < 1 -
\thetavec_1^{2} \leq \coef \frac{\LRBound^{3/2}}{\numobs}$, 
the following relation 
\begin{align*}
\Term_2(\thetavec_1) \geq - 2 | \corvec_1 | \cdot | \thetavec_1 - 1
| & \stackrel{(i)}{\geq}   -2 | \corvec_1 | \cdot \left( 1 - \thetavec_1^{2} \right) 
 \stackrel{(ii)}{\geq}
-20 \coef  \frac{
  \LRBound}{\numobs^{3/2}}
\end{align*}
holds with probability at least 0.99. 
Here step (i) uses the fact that 
\begin{align*}
|\thetavec_1-1|=|1-\sqrt{1-(1-\thetavec_1^{2})}|\leq1-\thetavec_1^{2}
\qquad
\text{for all }\quad \thetavec_1\in[0,1],
\end{align*}
and step (ii) relies on 
Lemma~\ref{lemma:second-moment} and the constraint
$1-\thetavec_1^{2}\leq \coef \frac{\LRBound^{3/2}}{\numobs}$.

\item {\bf Lower bounding $ \min_{\thetavecRem \in \ConstraintPrime(\thetavec_1)} 
\{ \tfrac{1}{2} \|\thetavecRem\|_{2}^{2} - 2
\corvecRem^{\top} \thetavecRem \}$.}  
When $\thetavec_1 \in [0,1)$, 
the constraint set $\ConstraintPrime(\thetavec_1)$ has
  non-empty interior, and the minimization over $\thetavecRem$ is
  strictly feasible. In this case, strict duality holds so that
\begin{align*}
\min_{\thetavecRem \in \ConstraintPrime(\thetavec_1)} \left\{
\tfrac{1}{2} \|\thetavecRem \|_{2}^{2} - 2 \corvecRem^{\top}
\thetavecRem \right \} & = \max_{\dualVar \geq 0} \left\{ -\dualVar
(1-\thetavec_1^{2}) - \sum_{j=2}^{\numobs} \frac{( \corvec_{j}
  )^{2}}{\tfrac{1}{2} + \tfrac{\dualVar}{\mu_{j}} } \right\} \\
& \geq -\left[ \numobs (1-\thetavec_1^{2}) \right]^{-2/3}
(1-\thetavec_1^{2}) - \sum_{j=2}^{\numobs}
\frac{(\corvec_j)^{2}}{\tfrac{1}{2} + \tfrac{ \left( \numobs
    \left(1-\thetavec_1^{2} \right)\right)^{-2/3}}{\mu_{j}}} \\
& = -\tfrac{\left(1-\thetavec_1^{2}\right)^{1/3}}{\numobs^{2/3}} -
\sum_{j=2}^{\numobs} \tfrac{( \corvec_{j} )^{2}}{\tfrac{1}{2} +
  \tfrac{\left(\numobs\left(1-\thetavec_1^{2}\right)\right)^{-2/3}}{\mu_{j}}}.
\end{align*}
Here the second line arises from a particular choice of $\dualVar$,
namely $ \dualVar = \left(\numobs (1-\thetavec_1^{2})\right)^{-2/3}.$
Since $1 - \thetavec_1^{2} \leq \coef
\frac{\LRBound^{3/2}}{\numobs}$, we further have
\begin{align*}
- \frac{\left(1-\thetavec_1^{2}\right)^{1/3}}{\numobs^{2/3}} - 
\sum_{j=2}^{\numobs}\frac{( \corvec_{j} )^{2}}{\frac{1}{2} + 
\frac{\left(\numobs\left(1-\thetavec_1^{2}\right)\right)^{-2/3}}{\mu_{j}}} & 
\geq - \frac{\coef^{1/3} \LRBound^{1/2}}{\numobs} - 
\sum_{j=2}^{\numobs} \frac{( \corvec_{j} )^{2}}{\frac{1}{2} + 
\frac{\left(\coef \LRBound^{3/2}\right)^{-2/3}}{\mu_{j}}} \\
 & = - \frac{\coef^{1/3}\LRBound^{1/2}}{\numobs} - 
 \sum_{j=2}^{\numobs} \frac{( \corvec_{j} )^{2}}{\frac{1}{2} 
 + \frac{1}{\coef^{2/3} \LRBound \mu_{j}}}\\
 & \geq - \tilde C \frac{\coef^{1/3} \LRBound^{1/2}}{\numobs},
\end{align*}
where $\tilde C > 0$ is a constant. 
Here, since 
$\LRBound$ is sufficiently large, Lemma~\ref{lemma:effective-dimension}
guarantees that  
the last inequality holds with probability at least 0.9. 

\end{itemize}

Combining the two cases above, we arrive at the conclusion 
that for any 
$1 - \thetavec_1^{2} \leq \coef \frac{\LRBound^{3/2}}{\numobs}$,
\begin{align*}
\keyfunc (\thetavec_1) \geq -20 \coef  \frac{
  \LRBound}{\numobs^{3/2}} 
- 2 C \frac{\log \numobs}{ \numobs } 
- \tilde C  \frac{\coef^{1/3} \LRBound^{1/2}}{\numobs}.
\end{align*}
Under the assumptions that $\LRBound \geq C_1 (\log \numobs)^2$ 
and $\numobs \geq C_2 \LRBound^{3/2}$ 
for some sufficiently large constants $C_1, C_2 > 0$, we can choose $\coef$ 
sufficiently small so as to make sure that 
\begin{align*}
\keyfunc(\thetavec_1) \geq  - \cstar
\frac{\LRBound^{1/2}}{\numobs} \qquad \text{for all }\quad 
1 - \thetavec_1^{2} \leq \coef \frac{\LRBound^{3/2}}{\numobs}.
\end{align*}


\section{Proofs of the bounds~\eqref{EqnElementary}}
\label{AppProofElementary}

By definition, any function $h \in \FuncClassStar$ obeys
$
  \hilnorm{h}  \leq 3 \radius.
$
In terms of the expansion $h = \SUMJ \theta_j \phi_j$, this constraint
is equivalent to the bound $\SUMJ \theta_j^2/\mu_j \leq 9 \radius^2 $.  In
addition, the constraint $\|h\|_{\Target} \leq r$ implies that $\SUMJ
\theta_{j}^2 \leq r^2$.  In conjunction, these two inequalities
imply that
\begin{align*}
\SUMJ \frac{\theta_{j}^2 }{ \min\{r^2,\mu_{j} \radius^2 \}} \leq 10,
\end{align*}
as claimed in inequality~\eqref{eq:fact}.

We now use this inequality to establish the
bound~\eqref{eq:Linf-bound}.  For any $x \in \Xspace$, we have
\begin{align*}
|h(x)| = \Big| \SUMJ \theta_{j} \phi_{j}(x) \Big| & = \Big| \SUMJ
\frac{ \theta_{j} }{ \sqrt{ \min\{ r^2, \mu_{j} \radius^2 \} } } \cdot \sqrt{
  \min\{r^2,\mu_{j} \radius^2 \}} \phi_{j}(x) \Big| \\
& \stackrel{(i)}{\leq} \sqrt{ \SUMJ \frac{ \theta_{j}^2 }{
    \min\{r^2,\mu_{j} \radius^2 \}} } \cdot \sqrt{ \SUMJ \min\{r^2,\mu_{j} \radius^2\}
  \phi_{j}^2(x)} \\
& \stackrel{(ii)}{\leq} \sqrt{10 \SUMJ \min\{ r^2,\mu_{j} \radius^2
  \}}.
\end{align*}
Here step (i) uses the Cauchy--Schwarz inequality, whereas step (ii)
follows from the previous claim~\eqref{eq:fact} and the assumption
that $|\phi_{j}(x)| \leq 1$ for all $j \geq 1$.


\newcommand{\MfunNew}{\Mfun^{\mathsf{new}}}
\newcommand{\rcritNew}{\rcrit^{\mathsf{new}}}
\section{Performance guarantees for LR-reweighted KRR}
\label{SecGeneralNoise}
In this section, we present the performance guarantee for the 
LR-reweighted KRR estimate with truncation for all ranges of 
$\sigma^2$. 

Similar to the large noise regime, we define 
\begin{align}
\label{eq:M-function-new}  
\MfunNew(\delta) & \coloneqq \Ccon \sqrt{ \tfrac{\sigma^2 \VarBound
    \log^3(\numobs)}{ \numobs} \KerComplex(\delta, \mu) } \; \left(
\sqrt{\tfrac{\KerComplex(\delta, \mu)}{\sigma^2}} + 1 \right).
\end{align}
Our theorem applies to any solution $\rcritNew >
0 $ to the inequality $\MfunNew(\delta) \leq \delta^2 / 2$.

\begin{theorem}
  \label{ThmReweightedKRRNew}
Consider a kernel with sup-norm bounded
eigenfunctions~\eqref{EqnEigenBounded}, and a source-target pair with
$\Exs_\Source[\likeratio^2(X)] \leq V^2$.  Then the estimate $\fhatrw$
with truncation $\TruncLevel = \sqrt{\numobs \VarBound}$ and
regularization $\lambda \radius^2 = \rcrit^2 / 3$ satisfies the bound
\begin{align}
\|\fhatrw - \TrueFun \|_{\Target}^2 \leq \rcrit^2
\end{align}
with probability at least $1 - c \: \numobs^{-10}$.  
\end{theorem}
\begin{proof}
Inspecting the proof of Theorem~\ref{ThmReweightedKRR} (in particular, 
equation~\eqref{eq:M-original-bound}), one has with high probability that 
\begin{align*}
\sup_{ \substack{g \in \GoodFunc(\rcrit)}} \Big \{ \| g - \TrueFun\|_{\Target}^2 + \frac{1}{\numobs}
\sum_{i=1}^{\numobs} \TruncLike{x_i} \big[\big( \TrueFun(x_i) - y_{i}
  \big)^2 - \big( g(x_i) - y_{i} \big)^2 \big] \Big \} \leq \MfunNew(\rcrit).
\end{align*} 
Repeating the analysis in Section~\ref{SecProofThmReweightedKRR} 
with $\regCrit = \rcrit$ yields the desired claim. 
\end{proof}


\section{Expectation bounds for KRR estimates}

In this section, we derive expectation bounds as counterparts to our
previous high probability upper bounds on the KRR estimates.
In Section~\ref{sec:cor-KRR-bounded-upper}, we present an expectation bound
for instances with bounded likelihood ratios, essentially as
a consequence of our previous high probability statement, given in Theorem~\ref{ThmKRR}.
Similarly, in Section~\ref{sec:cor-KRR-unbounded-upper}, we present
an expectation bound for instances which have possibly unbounded likelihood ratios,
but for which the second moment of the likelihood ratios is bounded.
Again, this can be seen as an extension of our previous high-probability
statement on the truncated, reweighted KRR estimator, as stated in Theorem~\ref{ThmReweightedKRR}. 

\subsection{Bounded likelihood ratio}
\label{sec:cor-KRR-bounded-upper}

\begin{theorem}\label{thm:KRR-bounded-exp}
Consider a covariate-shifted regression problem with likelihood ratio
that is $\LRBound$-bounded~\eqref{EqnLRBound} over a Hilbert space
with a $\kappa$-uniformly bounded kernel~\eqref{EqnKerBound}.  There are
universal constants $c_1, c_2 > 0$ such that if
$\lambda \geq c_1 \tfrac{\kappa^2 \log \numobs}{\numobs}$, the KRR estimate
$\fhatkrr$ satisfies the bound
\begin{align}
\label{EqnKRRBoundExp}
\Exp \big[\| \fhatkrr - \TrueFun \|_{\Target}^2\big]
\leq c_2  \Big\{ \lambda  \LRBound \radius^2 +
\frac{\sigma^2  \LRBound}{\numobs} \sum_{j=1}^{\infty}
  \frac{\mu_j }{ \mu_j + \lambda \LRBound }
  + \frac{\sigma^2}{\numobs}
  \Big\}.
\end{align}
\end{theorem}
\noindent Inspecting the proof, one may take $c_1 = 32, c_2 = \tfrac{519}{256}$. 
The proof of this result is presented in Section~\ref{sec:proof-cor-exp-bounded}.

An immediate consequence is the following result for regular kernels.
Note that it matches our lower bound (see Theorem~\ref{ThmLowerBound}),
apart from logarithmic factors. 

\begin{corollary}\label{cor:KRR-bounded-exp}

Suppose $\sigma^2 \geq \kappa^2$ and $\radius = 1$. For any $\LRBound \geq 1$ and any pair $(P, Q)$ with $\LRBound$-bounded likelihood
ratio~\eqref{EqnLRBound}, any orthonormal basis $\{\phi_j\}_{j \geq 1}$ of $L^2(Q)$,
and any regular sequence of kernel eigenvalues $\{\mu_j\}_{j \geq 1}$, there exist a
universal constant $C > 0$ such that 
\begin{align}
\label{EqnKRRBoundExpCor}
\Exp \big[\| \fhatkrr - \TrueFun \|_{\Target}^2\big] & \leq C \;
\inf_{\delta  > 0} \Big\{  \delta^2  + 
\sigma^2  \LRBound  d(\delta)  \frac{ \log \numobs }{\numobs} 
  \Big\},
\end{align}
where above $\lambda = \delta_\numobs^2$ where $\delta_\numobs^2 = c \frac{\sigma^2 B d(\delta_\numobs) \log \numobs}{\numobs}$
for a universal constant $c > 0$. 
\end{corollary}

\begin{proof}
Following the proof of Corollary~\ref{CorKRR}, 
we obtain from the KRR risk bound of Theorem~\ref{thm:KRR-bounded-exp}, 
\begin{equation}\label{ineq:krr-final-exp-bounded}
\Exp
\big[
\|\fhatkrr - \TrueFun\|_{\Target}^2\big]
\Big]
\leq
C_1 \Big\{ \delta^2  + 
\sigma^2  \LRBound  d(\delta)  \frac{ \log \numobs }{\numobs} 
  \Big\},
  \quad \mbox{where} \quad \delta^2 = \lambda \LRBound,
\end{equation}
for any $\delta^2 \geq c_1 \LRBound \kappa^2 \tfrac{\log \numobs}{\numobs}$.
Adjusting constants so that $c \geq c_1$, our choice of $\delta_\numobs^2$ is valid
since $\sigma^2 \geq \kappa^2$ and $d(\delta_\numobs) \geq 1$.
Moreover, since $\delta^2$ is an increasing function of $\delta$, whereas
$d(\delta)$ is nonincreasing, under the choice of 
$\delta_\numobs^2 = c \frac{\sigma^2 B d(\delta_\numobs) \log \numobs}{\numobs}$,
we have
\begin{equation}\label{ineq:comparison-opt}
\Big\{ \delta_\numobs^2  + 
\sigma^2  \LRBound  d(\delta_\numobs)  \frac{ \log \numobs }{\numobs} 
  \Big\}
\leq
C_2 \inf_{\delta > 0}
\Big\{ \delta^2  + 
\sigma^2  \LRBound  d(\delta)  \frac{ \log \numobs }{\numobs} 
  \Big\},
\end{equation}
for a universal constant $C_2 > 0$. Note that this inequality completes the proof of the result,
with $C = C_1 C_2$. 
\end{proof}
\subsubsection{Proof of Theorem~\ref{thm:KRR-bounded-exp}}
\label{sec:proof-cor-exp-bounded}

Using Parseval's theorem and the optimality conditions for the KRR
problem as given in equation~\eqref{eqn:opt-conditions-KRR}, we have
$\Exp [\| \fhatkrr - \TrueFun \|_{\Target}^2] \leq \Exp[T_1] + \Exp[T_2]$ where
\[
\Term_1 \coloneqq \| \lambda ( \EmpCov_{\Source} + \lambda \Mmat^{-1}
)^{-1} \Mmat^{-1} \thetastar \|_2^2,  \quad  \mbox{and}  \quad 
\Term_2 \coloneqq  \| ( \EmpCov_{\Source} + \lambda \Mmat^{-1} )^{-1}
\big( \frac{1}{\numobs} \sum_{i=1}^{\numobs} \noise_i \phi (x_i)
\big) \|_2^2.
\]
Recall the event
\[
\Event(\lambda) \coloneqq \Big \{ \Mmat^{1/2} \EmpCov_{\Source}
\Mmat^{1/2} + \lambda \IdMat \succeq \frac{1}{2} \big ( \Mmat^{1/2}
\PopCov_{\Source} \Mmat^{1/2} + \lambda \IdMat \big ) \Big \},
\]
as defined in equation~\eqref{eq:good-event-bounded}.
We use this event to bound the two terms.
\paragraph{Bound for $\Term_1$}
Inspecting the proof of Theorem~\ref{ThmKRR} (specifically, see the proof of bound~\eqref{EqnKRRTermBounds}(a)),
it follows that $\Exp[\Term_1 \1_{\Event(\lambda)}] \leq 2 \lambda \LRBound \radius^2$.
On the other hand, from inequality (ii) of the proof of~\eqref{EqnKRRTermBounds}(a), it also holds that
\[
\Exp[\Term_1 \1_{\Event(\lambda)^c}] 
\leq \radius^2 \opnorm{\Mmat} \Prob(\Event(\lambda)^c) \leq \radius \kappa^2 \Prob(\Event(\lambda)^c).  
\]
The final inequality holds since $\opnorm{\Mmat} \leq \trace(\Mmat) = \Exp_Q[\sum_j \mu_j \phi_j^2(x)] \leq \kappa^2$.
Now, note that whenever $\numobs \lambda \geq 32 \kappa^2 \log \numobs$, by Lemma~\ref{LemBasicCov} we have that
\begin{align*}
\Exp[\Term_1 \1_{\Event(\lambda)^c}] 
&\leq \radius^2  \Prob(\Event(\lambda)^c) \\ 
&\leq 28 \lambda \radius^2 
\Big[ \big(\frac{\kappa^2}{\lambda}\big)^2\exp\Big(-\frac{\numobs \lambda}{16 \kappa^2}\Big)\Big]  \\
&\leq \frac{7}{256} \lambda \radius^2.
\end{align*}
Putting the pieces together, we obtain
\begin{equation}\label{ineq:t1-upper-cor-bounded}
\Exp[\Term_1] \leq \frac{519}{256} \lambda \radius^2. 
\end{equation}
\paragraph{Bound for $\Term_2$}
By considering the expectation over $\noise_i$ conditional on the covariates
and following algebraic manipulations similar to the proof
of bound~\eqref{EqnKRRTermBounds}(b), we have
\[
\Exp[\Term_2] \leq \Exp[\widetilde \Term_2],
\quad \mbox{where} \quad
\widetilde \Term_2 \coloneqq
\trace \Big ( \frac{\sigma^2}{ \numobs } ( \EmpCov_{\Source} +
\lambda \Mmat^{-1} )^{-1} \Big ).
\]
Moreover, inspecting the proof of bound~\eqref{EqnKRRTermBounds}(b), we also have
\[
\Exp[\tilde \Term_2 \1_{\Event(\lambda)}]
\leq 2 \frac{\sigma^2 \LRBound }{\numobs} \sum_{j=1}^\infty \frac{\mu_j}{\mu_j + \lambda \LRBound}.  
\]
On the other hand, by bounding $( \EmpCov_{\Source} +
\lambda \Mmat^{-1} )^{-1} \preceq \lambda^{-1} \Mmat$, 
\[
\Exp[\tilde \Term_2 \1_{\Event(\lambda)^c}]
\leq \frac{\sigma^2}{\numobs} \frac{\kappa^2}{\lambda} \Prob(\Event(\lambda)^c)
\leq \frac{7}{256} \frac{\sigma^2}{\numobs}.
\]
The final inequality above is established in the same manner as in the proof
of the bound for $\Term_1$ above, when $\numobs \lambda \geq 32 \kappa^2 \log \numobs$.
Thus, combining the two bounds,
\[
\Exp[\Term_2] \leq 2 \frac{\sigma^2 \LRBound }{\numobs} \sum_{j=1}^\infty \frac{\mu_j}{\mu_j + \lambda \LRBound}
+ \frac{7}{256} \frac{\sigma^2}{\numobs}.
\]
\subsection{Unbounded likelihood ratio} 
\label{sec:cor-KRR-unbounded-upper}

\begin{theorem}\label{thm:KRR-unbounded-exp}
Suppose $\sigma^2 \geq \kappa^2$ and $\radius = 1$.
Consider a kernel with sup-norm bounded
eigenfunctions~\eqref{EqnEigenBounded}, and a source-target pair with
$\Exs_\Source[\likeratio^2(X)] \leq V^2$.
Then, for any orthonormal basis $\{\phi_j\}_{j \geq 1}$ of $L^2(Q)$
and any regular sequence of kernel eigenvalues $\{\mu_j\}_{j \geq 1}$,
there exists a universal constant $C > 0$ such that,
\begin{align}
\label{EqnKRRUnboundExpCor}
\Exp\Big[\|\fhatrw - \TrueFun \|_{\Target}^2 \Big] & \leq C \;
\inf_{\delta  > 0} \Big\{  \delta^2  + 
\VarBound    d(\delta)  \frac{ \log^3 \numobs }{\numobs} 
  \Big\}.
\end{align}
Above, $3 \lambda = \delta_\numobs^2$ where $\delta_\numobs^2$
satisfies the equation $\delta^2 = c \tfrac{\sigma^2 V^2 \log^3 \numobs}{\numobs}$
for a universal constant $c > 0$. 
\end{theorem}

Before giving the proof, we emphasize that---apart from logarithmic factors---this bound
is minimax optimal. 

\begin{proof}
By Theorem~\ref{ThmReweightedKRR}, there is an event $\Event$ which has probability
at least $1 -c \numobs^{-10}$ such that the truncated, reweighted estimator $\fhatrw$ satisfies 
\[
\|\fhatrw - \TrueFun \|_{\Target}^2 \leq c_1 \delta^2, 
\]
provided we select $\lambda \asymp \delta^2 \asymp \frac{\sigma^2 V^2 \log^3(\numobs) d(\delta)}{\numobs}$.
Note that under this choice of $\delta^2$, we have
\[
\delta^2 \asymp \inf_{\delta > 0} \Big\{ \delta^2 + \frac{\sigma^2 V^2 \log^3(\numobs) d(\delta)}{\numobs} \Big\}. 
\]
Consequently, there is a constant $c_2 > 0$ such that
\begin{equation}\label{ineq:final-exp-unbounded}
\Exp\Big[\|\fhatrw - \TrueFun \|_{\Target}^2 \Big]
\leq c_2 \; \inf_{\delta > 0} \Big\{ \delta^2 + \frac{\sigma^2 V^2 \log^3(\numobs) d(\delta)}{\numobs} \Big\}
+ \Exp\Big[\|\fhatrw - \TrueFun \|_{\infty}^2 \1_{\mathcal{E}^c} \Big]. 
\end{equation}
By Cauchy-Schwarz,
\[
\|\fhatrw - \TrueFun\|_{\infty}^2 \leq \kappa^2 \|\fhatrw - \TrueFun\|_{\Hilbert}^2
\leq 2 \kappa^2 (1 + \|\fhatrw\|_{\Hilbert}^2). 
\]
Applying the optimality condition of the reweighted estimator $\fhatrw$, we have
\[
\lambda \|\fhatrw\|_{\Hilbert}^2 \leq \lambda + \sqrt{\numobs V^2} \frac{1}{\numobs} \sum_{i=1}^\numobs \noise_i^2. 
\]
Therefore, combining the previous two displays, 
\[
\|\fhatrw - \TrueFun\|_{\infty}^2  \leq 2 \kappa^2 \Big(2 + \frac{\sqrt{\numobs V^2}}{\lambda} \frac{1}{\numobs} \sum_{i=1}^\numobs \noise_i^2\Big). 
\]
It then follows by Cauchy-Schwarz and the sub-Gaussianity of $\noise_i$, that for some constant $c_3 > 0$, 
\begin{align*}
\Exp\Big[\|\fhatrw - \TrueFun \|_{\infty}^2 \1_{\mathcal{E}^c} \Big]
&\leq c_3 \Big(\frac{\kappa^2}{\numobs^{10}} + \frac{\sigma^2 V^2}{\lambda \numobs^4} \kappa^2\Big) \\
&\stackrel{{\rm (i)}}{\leq} c_3 \frac{\sigma^2 V^2 }{\numobs} \Big(\frac{1}{\numobs^{9}} + \frac{\kappa^2}{\lambda} \frac{1}{\numobs^3}\Big) \\
&\stackrel{{\rm (ii)}}{\leq} c_3 \frac{\sigma^2 V^2}{\numobs} \Big( \frac{1}{\numobs^9} + \frac{c_4}{\numobs^2}\Big) \\
& \stackrel{{\rm (iii)}}{\leq} c_5 \frac{\sigma^2 V^2}{\numobs}
\end{align*}
Above, inequality (i) uses $\sigma^2 \geq \kappa^2$ and $V^2 \geq 1$.
Inequality (ii) uses the fact that $\lambda \asymp \delta^2 \asymp
\frac{\sigma^2 V^2 \log^3(\numobs) d(\delta)}{\numobs} \gtrsim \frac{\kappa^2}{\numobs}$.
Finally, inequality (iii) follows by defining $c_5 \geq c_3 (1 + c_4)$.
This bound furnishes the result, since by applying it to the inequality~\eqref{ineq:final-exp-unbounded},
we obtain the result with $C = c_2 + c_5$. 
\end{proof}


\section{Performance of unweighted KRR with unbounded likelihood ratios}
In this section, we present the performance guarantee of the unweighted KRR estimator 
when the likelihood ratios are unbounded. 

\begin{theorem}\label{thm:KRR-unbounded}
Consider a covariate-shifted regression problem with likelihood ratios obeying 
$\Exs_\Source[\likeratio^2(X)] \leq V^2$. 
Then
for any $\lambda \geq 10 \kappa^2 / \numobs$, the KRR estimate
$\fhatkrr$ satisfies the bound
\begin{align}
\label{EqnKRRBound-unbounded}
\| \fhatkrr - \TrueFun \|_{\Target}^2 \leq 2 \sqrt{ \lambda \VarBound \kappa^2} \radius^2 
+ 40 \frac{\sigma^2 \log \numobs }{ \numobs } \cdot  \frac{ \kappa^2 }{ \lambda }
\end{align}
with probability at least $1 - 28 \; \tfrac{\kappa^2}{ \lambda } e^{-\frac{
    \numobs \lambda }{16 \kappa^2}} - \frac{1}{\numobs^{10}}$. 
\end{theorem}

Simple algebra shows that the unweighted KRR estimator is still consistent for 
estimation under covariate shift, with a rate of $(\frac{\sigma^2 \VarBound }{ \numobs} )^{1/3}$ 
(ignoring $\kappa^2$ and $\log$ factors). However, unfortunately, this is far from optimal. 

\subsection{Proof of Theorem~\ref{thm:KRR-unbounded}}
In view of the proof of Theorem~\ref{ThmKRR}, we know that 
\begin{align}
\| \fhatkrr - \TrueFun \|_{\Target}^2 & \leq 4 \lambda \radius^2 \opnorm{ \Mmat^{1/2} ( \Mmat^{1/2}
  \PopCov_{\Source} \Mmat^{1/2} + \lambda \IdMat )^{-1} \Mmat^{1/2} } \nonumber \\
  &\quad + 
  40 \frac{\sigma^2 \log \numobs }{ \numobs } \trace \Big(
\Mmat^{1/2} (\Mmat^{1/2} \PopCov_{\Source} \Mmat^{1/2} + \lambda
\IdMat )^{-1} \Mmat^{1/2} \Big)
\end{align}
holds with probability 
at least $1 - 28 \; \tfrac{\kappa^2}{ \lambda } e^{-\frac{
    \numobs \lambda }{16 \kappa^2}} - \frac{1}{\numobs^{10}}$. The proof is 
    finished with the help of the following two bounds:
  \begin{subequations}
  \begin{align}
  	\opnorm{ \Mmat^{1/2} ( \Mmat^{1/2}
  \PopCov_{\Source} \Mmat^{1/2} + \lambda \IdMat )^{-1} \Mmat^{1/2} } & \leq 
  \frac{1}{2} \sqrt{ \frac{ \VarBound  \kappa^2 }{ \lambda } }; \label{eq:unbounded-op}\\
  	\trace \Big(
\Mmat^{1/2} (\Mmat^{1/2} \PopCov_{\Source} \Mmat^{1/2} + \lambda
\IdMat )^{-1} \Mmat^{1/2} \Big) & \leq \frac{ \kappa^2 }{ \lambda } \label{eq:unbounded-trace}.
  \end{align}
 \end{subequations}

\paragraph{Proof of the bound~\eqref{eq:unbounded-trace}: } 
Note that $\Mmat^{1/2} ( \Mmat^{1/2}
  \PopCov_{\Source} \Mmat^{1/2} + \lambda \IdMat )^{-1} \Mmat^{1/2} \preceq 
  \lambda^{-1} \Mmat $. We therefore have 
  \begin{align*}
  \trace \Big(
\Mmat^{1/2} (\Mmat^{1/2} \PopCov_{\Source} \Mmat^{1/2} + \lambda
\IdMat )^{-1} \Mmat^{1/2} \Big) \leq \trace (
\lambda^{-1} \Mmat ) \leq \frac{ \kappa^2 }{ \lambda },
  \end{align*}
  where the last relation uses the fact that $\trace (\Mmat) \leq \kappa^2$.
  
  \paragraph{Proof of the bound~\eqref{eq:unbounded-op}: } 
  We first make the observation that the bound~\eqref{eq:unbounded-op} is 
  equivalent to  
  \begin{align}\label{eq:unbounded-op-new}
  \PopCov_{\Source} + \lambda \Mmat^{-1} \succeq 2 \sqrt { \frac{ \lambda }{
   \VarBound \kappa^2 } } \IdMat.
  \end{align}
  Therefore from now on, we focus on establishing the bound~\eqref{eq:unbounded-op-new}.
  Take an arbitrary vector $\theta$ with $\| \theta \|_2 = 1$. We have 
  \begin{align*}
  1 = \| \theta \|_2^2 & \stackrel{(i)}{=}  \Exs_{\Target} [ ( \theta^\top \phi(X) )^2  ] 
  \stackrel{(ii)}{=} \Exs_{\Source} [ \likeratio(X) \cdot ( \theta^\top \phi(X) )^2  ] \\
  & \stackrel{(iii)}{\leq } \sqrt{ \Exs_{\Source} [  \likeratio^2 (X) ] } \cdot 
  \sqrt{ \Exs_{\Source} [   ( \theta^\top \phi(X) )^4 ] } \\
  & \stackrel{(iv)}{=} \sqrt{ \VarBound} \cdot \sqrt{ \Exs_{\Source} [   ( \theta^\top \phi(X) )^4 ] } .
  \end{align*}
 Here, the identity $(i)$ follows from the fact that $\Exs_\Target [ \phi(X) \phi (X)^\top ] = 
 \IdMat$, the relation $(ii)$ changes the measure from $\Target$ to $\Source$, the 
 inequality $(iii)$ is due to Cauchy-Schwarz, and the equality $(iv)$ uses the definition 
 of $\VarBound$. Apply the Cauchy-Schwarz inequality again to obtain 
 \begin{align*}
 	( \theta^\top \phi(X) )^2 \leq \| \Mmat^{-1/2} \theta \|_2^2 \cdot \| \Mmat^{1/2} \phi(X) \|_2^2 
	\leq \kappa^2 \| \Mmat^{-1/2} \theta \|_2^2, 
 \end{align*}
 where the second inequality relies on the fact that $\sup_x  \| \Mmat^{1/2} \phi(x) \|_2^2 
 \leq \kappa^2$. Take the above inequalities together to yield 
 \begin{align*}
 	\Exs_{\Source} [   ( \theta^\top \phi(X) )^2 ] \geq \frac{1 }{ \VarBound \kappa^2 
	\cdot ( \theta^\top  \Mmat^{-1} \theta ) } \qquad \text{for any }\theta \text{ with } \| \theta \|_2 = 1.
 \end{align*}
 As a result, one has 
 \begin{align*}
 	\theta^\top ( \PopCov_{\Source} + \lambda \Mmat^{-1} ) \theta \geq 
	\frac{1 }{ \VarBound \kappa^2 
	\cdot ( \theta^\top  \Mmat^{-1} \theta ) } + \lambda \theta^\top  \Mmat^{-1} \theta \geq 
	 2 \sqrt { \frac{ \lambda }{
   \VarBound \kappa^2 } }.
 \end{align*}
 Since this inequality holds for any unit-norm $\theta$, we establish the claim~\eqref{eq:unbounded-op-new}.

\section{Auxiliary lemmas}

The following lemma provides concentration inequalities for the sum of
independent self-adjoint operators, which appeared in the
work~\cite{minsker2017some}.
\begin{lemma}
\label{lemma:bernstein-operator}
Let $\Zmat_1, \Zmat_2, \ldots, \Zmat_\numobs$ be i.i.d.~self-adjoint
operators on a separable Hilbert space.  Assume that $\Exp [\Zmat_1] =
\bm{0}$, and $\opnorm{\Zmat_1} \leq L$ for some $L > 0$. Let $\Vmat$ be a
positive trace-class operator such that $\Exp [\Zmat_1^2] \preceq
\Vmat$, and $\opnorm{ \Exp [\Zmat_1^2] } \leq R $. Then one has
\begin{align*}
\Prob \Big( \opnorm{ \frac{1}{\numobs} \sum_{i=1}^{\numobs} \Zmat_{i}
} \geq t \Big) \leq \frac{28 \trace(\Vmat) }{ R } \cdot
\exp \Big( -\frac{  \numobs t^2 / 2}{ R + L t / 3}
\Big),\qquad \text{for all }t \geq \sqrt{ R / \numobs} +
L / (3 \numobs).
\end{align*}
\end{lemma}

Next, we turn attention to bounding the maxima of empirical processes.
Let $X_1, X_2, \ldots, X_\numobs$ be independent random variables.
Let $\FuncClass$ be a countable class of functions uniformly bounded
by $b$. Assume that for all $i$ and all $f \in \FuncClass$, $\Exp [f
  (X_i)] = 0$.  We are interested in controlling the random variable
$Z \coloneqq \sup_{f \in \FuncClass} \sum_{i=1}^{\numobs} f(X_i)$, for
which the variance statistics $v^2 \coloneqq \sup_{f \in \FuncClass}
\Exp [ \sum_{i=1}^{n} ( f(X_i) )^2 ]$ is crucial.  Now we are in position to state
the classical Talagrand's concentration inequalities; see the
paper~\cite{klein2005concentration}.

\begin{lemma}
\label{lemma:talagrand}
For all $t > 0$, we have
\begin{align}
  \Prob (Z \geq \Exp[Z] + t) \leq \exp \Big( -\frac{t^2}{ 2 (v^2
    + 2v \Exp[Z] ) + 3 v t} \Big).
\end{align}
\end{lemma}


\end{document}